\documentclass[12pt]{amsart}
\usepackage{xcolor}
\usepackage{mathtools}
\usepackage[latin2]{inputenc}
\usepackage{amsmath, amssymb}
\usepackage{capt-of}

\usepackage[
  hmarginratio={1:1},     
  vmarginratio={1:1},     
  textwidth=430pt,        
  heightrounded,          
]{geometry}

\usepackage{hyperref}

\usepackage{enumerate}
\usepackage{tikz}
\usepackage{pgfplots}
\pgfplotsset{compat=newest}
\usepgfplotslibrary{fillbetween}
\mathtoolsset{showonlyrefs,showmanualtags}

\newtheorem{theorem}{Theorem}
\newtheorem{lemma}{Lemma}
\newtheorem{proposition}{Proposition}
\newtheorem{corollary}{Corollary}
\newtheorem{algorithm}{Algorithm}
\theoremstyle{definition}
\newtheorem{definition}{Definition}
\theoremstyle{remark}
\newtheorem{remark}{Remark}

\newcommand{\volm}{\operatorname{Vol}_{M}}

\allowdisplaybreaks

\title[Large data limit of the MBO scheme]{Large data limit of the MBO scheme for data clustering: convergence of the dynamics}

\author{Tim Laux}
\address[Tim Laux]{Institut f\"ur angewandte Mathematik, Universit\"at Bonn, Endenicher Allee 60, 53115 Bonn, Germany}
\email{tim.laux@hcm.uni-bonn.de}

\author{Jona Lelmi}
\address[Jona Lelmi]{Institut f\"ur angewandte Mathematik, Universit\"at Bonn, Endenicher Allee 60, 53115 Bonn, Germany}
\email{lelmi@hcm.uni-bonn.de}

\begin{document}

\maketitle

\begin{abstract}

 We prove that the dynamics of the MBO scheme for data clustering 
 converge to a viscosity solution to mean curvature flow. 
 The main ingredients are (i) a new abstract convergence result based on quantitative estimates for heat operators and (ii) the derivation of these estimates in the setting of random geometric graphs.
 
 To implement the scheme in practice, two important parameters are the number of eigenvalues for computing the heat operator and the step size of the scheme. 
 The results of the current paper give a theoretical justification for the choice of these parameters in relation to sample size and interaction width.

\medskip

\noindent \textbf{Keywords:} Graph MBO, clustering, semi-supervised learning, continuum limits, viscosity solutions.

  \medskip

\noindent \textbf{Mathematical Subject Classification (MSC2020)}:
	35D40 (primary), 
	53Z50, 
	35R01, 
	35R02. 
\end{abstract}

\section{Introduction}\label{sec:intro}
The MBO scheme was originally introduced by Merriman, Bence and Osher \cite{Merriman1992, Merriman1994} as a numerical method to approximate evolution by mean curvature flow. 
More recently, Bertozzi et al.\ adapted the scheme to problems in data science such as data clustering \cite{Gennip2014, Merkurjev2013, Merkurjev2014}. 
Due to its conceptual simplicity, the MBO scheme is an efficient and robust algorithm for such tasks. 
Let us recall how the scheme works in the simple case of two classes, i.e.\ when the goal is to split a dataset $V = \{x_1, . . . , x_n\}$ into two subsets. 
Let $G=(V, W)$ be a weighted graph with vertex set $V$ and weight matrix $W$. 
Let $\Delta$ be a suitable graph Laplacian. 
Assume that $\chi^0: V \to \{0,1\}$ encodes an initial guess for the clustering. 
After choosing a step-size $h>0$ and the number of iterations $N \in \mathbf{N}$ that we want to run, for $0 \le l \le N-1$ define inductively a new clustering $\chi^{l+1}:V\to\{0,1\}$ by performing the following two steps:
\begin{enumerate}
\item \textbf{Diffusion}. For $t >0$ define
\begin{equation}
u^l(t) := e^{-t\Delta}\chi^l.
\end{equation}
\item \textbf{Thresholding}. Update the clustering by setting
\begin{equation}
\left\{ \chi^{l+1} = 1 \right\} = \left\{ u^l(h) \ge \frac{1}{2}\right\}.
\end{equation}
\end{enumerate}
By a result of Esedo\={g}lu and Otto in \cite{Esedoglu2015}, $\chi^{l+1}$ solves
\begin{equation}
\chi^{l+1} \in \underset{u:V \to [0,1]}{\operatorname{argmin}} \left\{ E_G^h(u) - E_G^h(u-\chi^{l})\right\},
\end{equation}
where $E_G^h$ is the thresholding energy on $G$ and is defined for $v: V \to [0, 1]$ as
\begin{equation}
E_G^h(v) := \frac{1}{\sqrt{h}} \langle (1-v), e^{-h\Delta}v \rangle_V,
\end{equation}
with $\langle \cdot, \cdot \rangle_V$ denoting an inner product on $V$  defined so that $\Delta$ becomes self-adjoint. 
In our recent work \cite{LauxLelmiI} we presented the first rigorous study of the large-data limit of the MBO scheme in data clustering.
More precisely, given a sequence of random geometric graphs $G_n = (V_n, W_n)$ -- i.e.\ such that $V_n = \{ X_1, . . . , X_n\}$ for a family $\{X_i\}_{i=1}^{+\infty}$ of iid random points $X_i \in M$, for a $k$-dimensional closed Riemannian submanifold $M \subset \mathbf{R}^d$ -- we studied the $\Gamma$-convergence of the family $\{E_{G_n}^h\}_{n \in \mathbf{N}, h > 0}$. 
When the number of iterations of the MBO scheme is very large, its outcome can be thought of as a local minimizer of the thresholding energy, and thus our $\Gamma$-convergence result says that this will be qualitatively close to a local minimizer of a suitable variational problem in the continuum. 
As the selection of the local minimizer strongly depends on the dynamics of the \textit{gradient descent} followed by the algorithm, the next question is to study the convergence properties of said dynamics. 
This is the content of the present paper: we study the convergence of the dynamics of the MBO scheme in the two-class setting. 
In general -- i.e.\ when the number of classes to cluster into is greater than two -- this is a much harder problem. 
In the two-class setting the task is easier because one can use the comparison principle for mean curvature flow, and thus the viscosity solutions setting. After the first works on viscosity solutions \cite{Crandall1992}, the machinery has proven to be a solid way to develop a theory of weak solutions for many problems satisfying a maximum principle -- and its use is the base for many fundamental contributions in geometric PDEs \cite{Chen1991, Evans1991} numerical analysis \cite{Barles1995, Ishii1999} and, more recently, for new results in theoretical data science \cite{Calder2019, Calder2019a, Bungert2021}.

We will always work with a sequence of weighted geometric graphs $G_n = (V_n, W_n)$, where the vertex sets $V_n$ are defined by $V_n := \{x_i\}_{i=1}^{n}$, where $\{x_i\}_{i=1}^{+\infty}$ is a sequence of points on $M \subset \mathbf{R}^d$ -- a $k$-dimensional closed Riemannian submanifold of $\mathbf{R}^d$ -- and the weight matrix $W_n$ is obtained in the by-now-standard way of weighting the edge between two distinct points with a suitable non-increasing function of the Euclidean distance between them, properly rescaled by a localization parameter $\epsilon_n > 0$, see Section \ref{sec:mainres} for the precise construction. 
In this setting, we study the convergence of the sequence of dynamics of the MBO scheme on these graphs as the data size $n$ goes to infinity.

This paper can be conceptually thought of as divided into two main results: in the first one, Theorem~\ref{thm:conditional_conv}, we work in an abstract setting. 
First, in the MBO scheme, we replace the heat operators on the graphs with abstract operators $S_n: (0, +\infty) \times \mathcal{V}_n \to \mathcal{V}_n$ which are linear in the second variable (here $\mathcal{V}_n$ is the space of real-valued functions defined on the vertex set $V_n$) and we show if the sequence $\{S_n\}_{n \in \mathbf{N}}$ approximates well-enough the heat kernel corresponding to a weighted Laplace--Beltrami operator on the manifold, then we have convergence of the dynamics of the MBO scheme on the graphs to the viscosity solution of mean curvature flow on the manifold. 
The conditions that the operators $\{S_n\}$ have to satisfy are three: \eqref{item:monotonicity} they should satisfy an approximate maximum principle, \eqref{item:estimate} they should approximate the action of the heat kernel on smooth function in a uniform sense, and \eqref{item:approx_one} their action on the constant function $\mathbf{1}$ should be close enough to the constant $\mathbf{1}$. All these properties are made quantitatively precise in Theorem~\ref{thm:conditional_conv}. 

The second main result is Theorem~\ref{thm:estimate_random_graphs} and its Corollary \ref{corollary: conv_geom_graphs}, where we check that \eqref{item:monotonicity}, \eqref{item:estimate} and \eqref{item:approx_one} are satisfied with high probability on random geometric graphs -- i.e.\ when the points $\{x_i\}_{i=1}^{+\infty}$ are sampled independently from a probability measure $\nu = \rho \volm \in \mathcal{P}(M)$, absolutely continuous with respect to the volume form -- and when $S_n$ are chosen to be the heat operators on the graphs or the operators obtained by cutting off frequencies higher than a threshold $K_n$ defined precisely in Item \eqref{eq:h_n_cond} in Theorem~\ref{thm:estimate_random_graphs}. 
Let us stress that the latter result is crucial for applications. 
Indeed, when one implements the MBO scheme on a large dataset, computing the full heat kernel is intractable, and thus one usually works with an approximate version of it obtained by cutting off high frequencies in precisely the way described above. 
Our result gives a solid mathematical justification for this procedure, proving that the scheme converges in the large data limit to the viscosity solution to mean curvature flow provided the frequency cut-off is chosen according to $K_n \geq \left( \log(n) \right)^q$ where $q$ is a suitable positive real number and $n$ is the number of data points. 
We also notice that Theorem~\ref{thm:estimate_random_graphs} gives sufficient conditions on how to choose the length scale $\epsilon_n$ and the time-step size $h_n$ in order to ensure convergence of the scheme. 
In particular, the choice of $h_n$ is not anymore based solely on rules-of-thumb but has theoretical foundations. 
Previously, only a negative result ensuring pinning of the scheme was known~\cite[Theorem 4.2]{Gennip2014}. 
However, we point out that the conditions on $\epsilon_n$ and $h_n$ are only sufficient, but not sharp. 
Indeed, we expect that the convergence of the scheme should hold true whenever $\epsilon_n = o(h_n)$, while our conditions imply that $\epsilon_n = o(h_n^{3/2})$. 
The sharp rate $\epsilon_n = o(h_n)$ was verified in the simple setting of the deterministic two-dimensional regular grid $\mathbf{Z}^2$ in \cite{Misiats2016}, and is based on the explicit expression for the heat kernel on regular grids. 
But an extension to the general setting in which we are working requires a different strategy, see also the discussion in Remark~\ref{rem:comparison_old} to better understand how our result compares to the one in the simple setting of \cite{Misiats2016}.

Let us spend a few words on the strategy of the proofs of Theorem~\ref{thm:conditional_conv} and Theorem~\ref{thm:estimate_random_graphs}. 
For Theorem~\ref{thm:conditional_conv} we follow the general scheme of proof of Barles and Georgelin~\cite{Barles1995}, also used in~\cite{Misiats2016}. 
The authors prove convergence of the classical MBO scheme to a viscosity solution to mean curvature flow in the Euclidean space. 
Given a smooth open set $\Omega \subset M$, the idea is to prove that the upper semicontinuous envelope $u^*$ and the lower semicontinuous envelope $u_*$ of the piecewise constant in time interpolations of outcomes of the MBO scheme (with initial values $\Omega \cap G_n$) as defined in \eqref{eq:subs_def} and \eqref{eq:sups_def} are, respectively, a viscosity subsolution and a viscosity supersolution to mean curvature flow on the manifold. 
After doing that, one has to use the comparison principle  in Theorem~\ref{thm:comparison} to compare $u^*$ and $u_*$ with the unique viscosity solution $u$ to mean curvature flow with initial value $\Omega$ to show that $\operatorname{sign}_*(u) \le u_*$ and $\operatorname{sign}^*(u) \ge u^*$. 
In order to check that $u^*$ and $u_*$ are, respectively, a viscosity subsolution and a viscosity supersolution to mean curvaturue flow we have to adapt the strategy in \cite{Barles1995} to our setting: we need to  carefully identify admissible error terms for the argument of \cite{Barles1995}. 
The estimate in item \eqref{item:estimate} in Theorem~\ref{thm:conditional_conv} plays a central role in this, as well as the extension of the consistency step to weighted manifolds (Theorem~\ref{thm:consist}).  
Finally, to apply the comparison principle in Theorem~\ref{thm:comparison}, it is crucial to show an ordering of the initial values in the sense that $\operatorname{sign}_*(u(0, \cdot)) \le u_*(0, \cdot)$ and $\operatorname{sign}^*(u(0, \cdot)) \ge u^*(0, \cdot)$. 
We verify this in the general case of a weighted manifold by carefully checking that one iteration of the MBO scheme with step size $h$ produces a set whose normal distance from the previous one is of order $h$ (Theorem~\ref{thm:thresholding_one_step}). 
This issue seems to have been overlooked in the literature and we believe that our proof fills an important gap in the previous works, even in the Euclidean setting.

For Theorem~\ref{thm:estimate_random_graphs} we draw inspiration from~\cite{Dunson2021}. 
There, the authors work on a fixed graph with points sampled independently from a weighted manifold and consider the error in a uniform sense between the restriction of the manifold heat kernel to the graph and the operator obtained by considering the first $K$ frequencies of the graph heat kernel. 
Their estimate, however, cannot be applied in our setting because, since we want to take the number of data points to infinity, we have to be able to take the frequency cut-off $K$ to infinity together with them. 
For this reason, a careful interplay between the chosen rates of convergence for $K$, the step size $h$ and the localization parameter $\epsilon$ is needed. In Lemma \ref{lem:heat_kernel_estimate} we obtain a new estimate giving precise conditions on the relation between the frequency cut-off and the number of data points.
To get this, we make use of recent results on convergence of spectra of graph Laplacians~\cite{GarciaTrillos2020, Calder2022, Calder2022a}.

After its introduction in this setting, several authors have developed variants of the MBO scheme. For instance, volume-preserving MBO scheme \cite{Jacobsetal,Jacobs2016} -- a version of the algorithm developed by Jacobs, Merkurjev, and Esedo\={g}lu, where the number of points belonging to each class is invariant through iterations -- and poissonMBO \cite{Calder2020} -- a variant of the scheme for semi-supervised learning at low labeling rates introduced by Calder, Cook, Thorpe and Slep\v{c}ev. When there are just two classes to split the dataset into, we believe that the techniques developed in the present work may be suitably modified to extend the results to these variants of the scheme. One may need to combine our ideas with the techniques developed by Kim and Kwon in \cite{Kim2020}, where the authors develop a viscosity solution approach for volume preserving mean curvature flow in Euclidean space.

\medskip

The rest of the paper is organized as follows: in Section~\ref{sec:thescheme} we introduce some notation and the two versions of the MBO scheme that we study -- the classical one by Bertozzi et al. \cite{Merkurjev2013, Merkurjev2014, Gennip2014}, and a more practical one in which the heat operator in the diffusion step is modified by cutting off high frequencies. In Section~\ref{sec:mainres} we state the main results of the current paper: Theorem~\ref{thm:conditional_conv} gives sufficient conditions for the abstract MBO scheme in Algorithm~\ref{algo_abstract} to converge to a viscosity solution to mean curvature flow on a weighted manifold. In Theorem~\ref{thm:estimate_random_graphs} and its Corollary~\ref{corollary: conv_geom_graphs} we show that these conditions are satisfied for the two versions of the algorithm that we study. In Section~\ref{sec:mcf_drift} we introduce the notion of viscosity solution to mean curvature flow on a weighted manifold by simply extending well-known ideas and results in the literature for mean curvature flow on compact manifolds \cite{Ilmanen1992} and Euclidean spaces \cite{Chen1991, Evans1991, Ambrosio2000}. In Section \ref{sec:mbo_manifolds} we introduce the MBO scheme on a weighted manifold and we state Theorem~\ref{thm:thresholding_one_step}, which says that one iteration of MBO produces a set whose normal distance from the previous one is of order $h$, the chosen step-size. In this section, we also give an extension to weighted manifolds of the consistency step in the work of Barles and Georgelin \cite{Barles1995}. In Section~\ref{sec:proofs} we present the proofs of the results of the paper. In the \hyperlink{sec:appendix}{Appendix}, we collect some results about the behavior of the heat kernel on weighted manifolds and on the asymptotics of the spectra for graph Laplacians which are needed in the proofs.

\medskip

\textbf{Notation}. In the present work, we make extensive use of the Landau symbols $o$, $O$. To explain these, we let $\{a_\omega\}_{\omega \in \Omega}, \{b_{\omega}\}_{\omega \in \Omega}$ be two families of real numbers, with $b_\omega > 0$, indexed by $\omega \in \Omega \subset \mathbf{R}$. Let $\omega_0 \in \mathbf{R} \cup \{-\infty, +\infty\}$ be a limit point for the set $\Omega$, which will be clear from the context. We say that $a_\omega = O(b_\omega)$ if
\begin{equation}
\limsup_{\omega \to \omega_0} \frac{a_\omega}{b_\omega} < +\infty.
\end{equation}
We say that $a_\omega = o(b_\omega)$ if
\begin{equation}
    \lim_{\omega \to \omega_0} \frac{a_\omega}{b_\omega} = 0.
\end{equation}
We also alternatively write $a_{\omega} \lesssim b_\omega$ for $a_\omega = O(b_\omega)$ and $a_{\omega} \ll b_\omega$ for $a_\omega = o(b_\omega)$. In the following, usually $(\Omega, \omega_0)$ will be $(\mathbf{N}, +\infty)$ or $(\mathbf{R}^+, 0)$, and this will be clear from the context.

\section{The MBO scheme on graphs}\label{sec:thescheme}

In this section, we describe the MBO algorithm on graphs originally given by Bertozzi et al.\ in \cite{Merkurjev2013, Gennip2014, Merkurjev2014}. We refer to \cite{LauxLelmiI} for more information about its use in data clustering. We consider a weighted connected graph $G = (V, W)$ with $n$ vertices, with $W_{ii} = 0$ for every $i = 1, . . . , n$. For each vertex $x_i \in V, i \in \{1, . . . , n\}$, we can define
\begin{equation}
d(x_i) = \frac{1}{n}\sum_{j=1}^n w_{ij}.
\end{equation}
We define $D := \operatorname{diag}(d(x_1), . . . , d(x_n))$. We let $\mathcal{V} := \left\{u| u: V \to \mathbf{R}\right\}$, the set of functions defined on $V$, which we endow this with the inner product
\begin{equation}
\langle u, v \rangle_{\mathcal{V}} := \frac{1}{n} \sum_{i=1}^n d(x_i) u(x_i) v(x_i).
\end{equation}
We define the random walk Laplacian $\Delta: \mathcal{V} \to \mathcal{V}$ as
the operator induced by the matrix
\begin{equation}
\Delta := \left( I - \frac{1}{n}D^{-1}W \right).
\end{equation}
One can check that $\Delta$ is non-negative and self-adjoint with respect to $\langle \cdot, \cdot \rangle_{\mathcal{V}}$, in particular, it has $n$ eigenvalues (counted with multiplicity) which we order in the following way
\begin{equation}
0 = \lambda^1 \le . . .  \le \lambda^n.
\end{equation}
We denote by $\{v^l\}_{1 \le l \le n}$ a basis of corresponding eigenvectors, orthonormal with respect to $\langle \cdot, \cdot \rangle_{\mathcal{V}}$. For $0 < K \le n$ we define a kernel $H^K: (0, +\infty) \times V \times V \to \mathbf{R}$ via
\begin{equation}
H^K(t, x, y) := \sum_{l=1}^K e^{-t\lambda^l}v^l(x)v^l(y)\frac{d(y)}{n}.
\end{equation}
The choice $K =  n$ corresponds to the heat kernel associated to $\Delta$, which is the unique function $H: (0, +\infty) \times V \times V \to \mathbf{R}$ with the property that for every $u_0 \in \mathcal{V}$, the function
\begin{equation}
u(t, x) := e^{-t\Delta}u_0(x) := \sum_{y \in V} H(t, x, y)u_0(y),\ x \in V,\ t > 0
\end{equation}
satisfies
\begin{equation}
\begin{cases}
\partial_t u = -\Delta u\ &\text{on}\ (0, +\infty) \times V,
\\ \lim_{t\downarrow 0} u(t, x) = u_0(x)\ &\text{on}\ V.
\end{cases}
\end{equation}
We are now ready to introduce the MBO scheme on graphs. 

\begin{algorithm}[MBO scheme]\label{algo_exact}
Fix a time-step size $h > 0$ and initial conditions $\chi^{0}: V \to \{0,1\}$. For each $l \in \mathbf{N}$ define inductively $\chi^{l+1}: V \to \{0,1\}$ as follows:
\begin{enumerate}
\item \textbf{Diffusion}. Define 
\begin{equation}
u^l := e^{-h\Delta}\chi^l.\nonumber
\end{equation}
\item \textbf{Thresholding}. Define $\chi^{l+1}$ by
\begin{equation}
\left\{ \chi^{l+1} = 1 \right\} = \left\{ u^l \ge \frac{1}{2} \right\}.\nonumber
\end{equation}
\end{enumerate}
\end{algorithm}
We then define the piecewise constant in time, right-continuous interpolation

\begin{equation}\label{eq:definition_of_approx}
\chi^{h, G}(t,x) = \chi^l(x)\ \text{for}\ t \in [lh, (l+1)h)\ \text{and}\ x \in V.
\end{equation}
 We are interested in understanding whether this approximation is consistent at the level of the evolution by mean curvature flow on the manifold. 
 
 In practice, computing the exact diffusion in the first step of the algorithm may be computationally intractable. For this reason, one usually implements the MBO scheme by considering only a smaller number of eigenvectors of the Laplacian, say $K$. In other words, one uses the following more efficient variant of MBO.

\begin{algorithm}[Approximate MBO scheme]\label{algo_approx}
Fix a time-step size $h > 0$ and initial conditions $\chi^{0}: V \to \{0,1\}$. For each $l \in \mathbf{N}$ define inductively $\chi^{l+1}: V \to \{0,1\}$ as follows:
\begin{enumerate}
\item \textbf{Diffusion}. Define 
\begin{equation}
u^l(x) := \sum_{y \in V} H^K(h, x, y)\chi^l(y).
\end{equation}
\item \textbf{Thresholding}. Define $\chi^{l+1}$ by
\begin{equation}
\left\{ \chi^{l+1} = 1 \right\} = \left\{ u^l \ge \frac{1}{2} \right\}.\nonumber
\end{equation}
\end{enumerate}
\end{algorithm}
Again, we then define the piecewise constant in time, right-continuous interpolation

\begin{equation}\label{eq:definition_of_approx_MBOAPP}
\chi^{h, G, K}(t,x) = \chi^l(x)\ \text{for}\ t \in [lh, (l+1)h)\ \text{and}\ x \in V.
\end{equation}
At present, the choice of $h$ and the exact value of $K$ to pick in order to get a good approximation of the MBO scheme is obtained by trial and error. In this work, under the standard \textit{manifold assumption}, we rigorously justify that an admissible regime to get a consistent result in the large-data limit is $K \ge (\log(n))^q$, $h \gg (\log(n))^{-\alpha}$ for some $q, \alpha > 0$ (see Theorem~\ref{thm:estimate_random_graphs} for the precise choices of $q, \alpha$).

\section{Main results}\label{sec:mainres}

Hereafter $M \subset \mathbf{R}^d$ is a $k$-dimensional closed Riemannian submanifold. We denote by $\{x_i\}_{i=1}^{+\infty}$ a sequence of points on $M$, and for each $n \in \mathbf{N}$ we define weighted graphs $G_n = (V_n, W_n)$ where the vertex set $V_n$ is given by $\{x_1, . . . , x_n\}$ and the adjacency matrix $W_n = (w_{ij}^{(n, \epsilon_n)})_{1 \le i,j \le n}$ is given by 

\begin{align*}
&w_{ii}^{(n, \epsilon_n)} = 0\ \text{for}\ 1 \le i \le n, 
\\ &w_{ij}^{(n, \epsilon_n)} = \frac{1}{\epsilon_n^k}\eta\left( \frac{\Vert x_i - x_j \Vert_d}{\epsilon_n} \right)\ \text{for}\ 1 \le i, j \le n,\ i \neq j.
\end{align*}
Here $\epsilon_n > 0$ are given length scales and $\eta: [0, +\infty) \to [0, +\infty)$ is a non-increasing function with support on the interval $[0,1]$, whose restriction to the interval $[0,1]$ is Lipschitz continuous. We define

\begin{align}
C_1 := \int_{\mathbf{R}^k} \eta(|y|_k) dy,\quad C_2 := \int_{\mathbf{R}^k} \eta(|y|_k)y_1^2 dy,\quad \kappa(\eta) := \frac{C_2}{2C_1}.
\end{align}
We also define, for every $x \in M$ and every $n \in \mathbf{N}$

\begin{equation}
d_{n}(x) := \frac{1}{n}\sum_{j=1}^n \frac{1}{\epsilon_n^k}\eta\left( \frac{\Vert x - x_j \Vert_d}{\epsilon_n}\right)\mathbf{1}_{\{x \neq x_j\}}.
\end{equation}
Note that, when $x = x_i$ for some $1 \le i \le n$, then $d_{n}(x)$ is the degree of the $i$-th node. We denote by $D_{n} := \operatorname{diag}(d_{n}(x_1),  . . . , d_{n}(x_n))$ the diagonal matrix of the degrees. The random walk Laplacian $\Delta_n$ is the linear operator induced by the $(n \times n)$-matrix given by

\begin{equation}
\Delta_n := \frac{1}{\epsilon_n^2}\left( I - \frac{1}{n}D_{n}^{-1}W_{n}\right).
\end{equation}
We denote by $\{ v_{n}^{l}\}_{1 \le l \le n}$ an orthonormal basis (with respect to the inner product $\langle \cdot, \cdot \rangle_{\mathcal{V}_n}$) made of eigenvectors for the Laplacian $\Delta_n$ corresponding to the eigenvalues $\{\lambda_{n}^{l}\}_{1\le l \le n}$, which are ordered in the following way

\begin{equation}
0 = \lambda_{n}^{1} \le \lambda_{n}^{2} \le . . .  \le \lambda_{n}^{n}.
\end{equation}
Like in Section \ref{sec:thescheme}, for every $0 < K \le n$ we define

\begin{equation}
H_n^{K}(t, x, y) = \sum_{l=1}^{K} e^{-t\lambda_n^l}v_n^l(x)v_n^l(y)\frac{d_n(y)}{n},
\end{equation}
and we set $H_n = H_n^n$ when $K = n$. Assume that we are given a sequence of operators $S_n: (0,+\infty) \times \mathcal{V}_n \to \mathcal{V}_n$ which are linear in the second variable, we then consider the following abstract version of the MBO scheme on the $n$-th graph.

\begin{algorithm}[Abstract MBO scheme]\label{algo_abstract}
Fix a time-step size $h_n > 0$ and initial conditions $\chi^{0, G_n}: V_n \to \{0,1\}$. For each $l \in \mathbf{N}$ define inductively $\chi^{l+1, G_n}: V_n \to \{0,1\}$ as follows:
\begin{enumerate}
\item \textbf{Diffusion}. Define 
\begin{equation}
u_n^l := S_n(h_n, \chi^{l, G_n}).
\end{equation}
\item \textbf{Thresholding}. Define $\chi^{l+1, G_n}$ by
\begin{equation}
\left\{ \chi^{l+1, G_n} = 1 \right\} = \left\{ u_n^l \ge \frac{1}{2} \right\}.\nonumber
\end{equation}
\end{enumerate}
\end{algorithm}
We then define $\chi^{h_n, G_n}: [0, +\infty) \times V_n \to \{0, 1\}$ by
\begin{equation}
\chi^{h_n, G_n}(t, x) := \chi^{l, G_n}(x),\ x \in V_n,\ t \in [lh_n, (l+1)h_n).
\end{equation}
For  convenience, we will mostly work with the $\{-1, 1\}$-valued functions
\begin{equation}
u^{h_n, G_n}(t, x) := 2\chi^{h_n, G_n}(t, x) - 1.
\end{equation}
We also define the upper and lower limits of the family $\{u^{h_n, G_n}\}_{n \in \mathbf{N}}$ as

\begin{align}
&\begin{aligned}\label{eq:subs_def}
u^*(t, x) := \sup\bigg\{ \limsup_{n \to +\infty} u^{h_n, G_n}(t_n, x_n)\bigg|&\ t_n >0,\ \lim_{n\to +\infty} t_n = t,
\\ &\ x_n \in G_n,\ \lim_{n \to +\infty} x_n = x\bigg\},
\end{aligned}
\\ &\begin{aligned}\label{eq:sups_def}
u_*(t, x) :=\, \inf\,\bigg\{\liminf_{n \to +\infty} u^{h_n, G_n}(t_n, x_n)\,\bigg|&\ t_n >0,\ \lim_{n\to +\infty} t_n = t,
\\ &\ x_n \in G_n,\ \lim_{n \to +\infty} x_n = x\bigg\}.
\end{aligned}
\end{align}
\\
Let $\xi > 0$ be a smooth function on the manifold $M$. Let $\Omega \subset M$ be an open set with smooth boundary $\Gamma_0$. We let $u: [0, +\infty) \times M \to \mathbf{R}$ be the unique viscosity solution of the level set formulation of the mean curvature flow with density $\xi$ (see Section \ref{sec:mcf_drift} for the details) with initial value $sd(\cdot, \Gamma_0) = d_M(x, \Omega^c) - d_M(x, \Omega)$, the signed distance function from $\Gamma_0$. For any $t>0$ we also define 
\begin{equation}\label{eq:notation_evolution_set}
\Omega_t := \left\{ x \in M\, |\ u(t, x) > 0\right\},\ \Gamma_t = \left\{ x \in M|\ u(t,x) = 0\right\}.
\end{equation}
Let us denote by $\Delta_{\xi}$ the weighted Laplacian on $M$ with weight $\mu := \xi \volm$, i.e.,

\begin{equation}
\Delta_{\xi} f = - \frac{1}{\xi}\operatorname{div}\left( \xi \nabla f\right)\quad \text{for}\ f \in C^{\infty}(M).
\end{equation}
Let $H: (0,+\infty) \times M \times M \to \mathbf{R}$ denote the corresponding heat kernel. 

Our first main result is the following conditional convergence of the abstract formulation of the MBO scheme.

\begin{theorem}\label{thm:conditional_conv}
Assume that:
\begin{enumerate}[(i)]
\item The operators $S_n$ satisfy the maximum principle up to errors $h_n^{3/2}$, i.e., for $n$ large enough and for each $u,v \in \mathcal{V}_n$ it holds
\begin{equation}
u \le v \Rightarrow S_n(h_n, u) \le S_n(h_n, v) + \left( \max_{V_n}|u| + \max_{V_n} |v|\right)O(h_n^{3/2}).
\end{equation}\label{item:monotonicity}
\item The operators $S_n$ approximate the heat operator on the manifold, i.e. 
there exists a constant $\kappa>0$ such that for every function $f \in C^{\infty}(M)$ we have

\begin{equation}\label{eq:needed_heat_est}
\max_{ x \in V_n} \left| S(h_n, f)(x) - e^{-h\kappa\Delta_{\xi}}f(x)\right| = (\sup |f| )\,o(\sqrt{h_n}) + \operatorname{Lip}(f)\,O(h_n^{3/2}).
\end{equation}
where the functions $o(\sqrt{h_n}), O(h_n^{3/2})$ are independent of $f$.\label{item:estimate}
\item The operators $S_n$ almost preserve the total mass in the sense that
\begin{equation}
\max_{x \in V_n} |S_n(h_n, \mathbf{1}_{G_n})(x) - 1| = O(h_n^{3/2}).
\end{equation}\label{item:approx_one}
\end{enumerate}
Then $u^*$ and $u_*$ defined in \eqref{eq:subs_def} and, respectively, \eqref{eq:sups_def} satisfy
\begin{align}
u_*(x,t) &= 1\quad \text{if}\ x \in \Omega_t,\label{eq:lowersol}
\\ u^*(x,t) &= -1\, \text{if}\ x \in (\Omega_t \cup \Gamma_t)^c.\label{eq:uppersol}
\end{align}
Here $\Omega_t$ and $\Gamma_t$ are defined as in \eqref{eq:notation_evolution_set}.
\end{theorem}

\begin{remark}\label{rem:comparison_old}
Let us compare Theorem \ref{thm:conditional_conv} with the work \cite{Misiats2016}, where the authors prove convergence of the dynamics of the graph MBO scheme to a viscosity solution to mean curvature flow in the case of regular, two-dimensional grids. More precisely, they work in the following setting: the manifold $M$ is the standard Euclidean plane $\mathbf{R}^2$, the sequence of graphs $G_n$ are given by $G_n := \epsilon_n \mathbf{Z}^2$ for a sequence of localization parameters $\epsilon_n \downarrow 0$ and for $(i,j), (l,m) \in \mathbf{Z}^2$ one sets

\begin{equation}\label{eq:infinite_matrix}
w_{(i,j),(l,s)}  = \begin{cases}
1\quad &\text{if}\ |i-l| + |m-j| = 1,
\\ 0\quad &\text{otherwise}.
\end{cases}
\end{equation} 
In this way we can define an infinite dimensional weight matrix $W_n$ whose entries are indexed by $\mathbf{Z}^2 \times \mathbf{Z}^2$ and are defined as \eqref{eq:infinite_matrix}. 
To put ourselves in a setting that is precisely the one we are working in we could actually work with $M = \mathbf{T}^2$, the $2$-dimensional torus, and the sequence of graphs $G_n \cap \mathbf{T}^2$, but to keep the discussion simple we prefer to continue this discussion in the precise setting of \cite{Misiats2016}. 
Let $v: \epsilon_n \mathbf{Z}^2 \to \mathbf{R}$ be a function which is zero outside a compact subset of $\mathbf{R}^2$. We denote by $S_n(t, v): [0, +\infty) \times \epsilon_n \mathbf{Z}^2 \to \mathbf{R}$ the solution to the heat equation on $G_n$ with initial value $v$, i.e., $u := S_n(t, v)$ solves

\begin{align}\label{eq:heat_eq_grid}
\begin{cases}
 \begin{aligned} \frac{d}{dt} u (t, (i, j)) = \frac{1}{\epsilon_n^2}\bigg[ &u(t, (i+1, j)) + u(t, (i-1, j)) 
\\ &+ u(t, (i, j+1)) + u(t, (i, j-1)) 
\\ &-4u(t, (i, j)) \bigg]\end{aligned}\quad &\text{for}\ (i,j) \in \epsilon_n\mathbf{Z}^2,
\\ u(0, (i,j)) = v((i, j))\quad &\text{for}\ (i,j) \in \epsilon_n\mathbf{Z}^2.
\end{cases}
\end{align}
In other words, $S_n(\cdot, v)$ is the heat operator on $G_n$ applied to $v$. By using Fourier analysis methods, it can be shown that for every $h > 0$ and every $(x_1, x_2) \in \epsilon_n\mathbf{Z}^2$

\begin{align*}
S_n(h, v)((x_1, x_2)) = \sum_{(i, j) \in \epsilon_n\mathbf{Z}^2} Q_{i-x_1}\left( \frac{2h}{\epsilon_n^2} \right) Q_{j - x_2}\left( \frac{2h}{\epsilon_n^2}\right)v( (i, j)),
\end{align*}
where
\begin{equation}\label{eq:one_dim_fourier}
Q_{l}(\alpha) := \frac{1}{2\pi} \int_{-\pi}^\pi \cos(l\xi) e^{\alpha(\cos(\xi)-1)}d\xi.
\end{equation}
Using the asymptotic expansions \cite[Proposition 3]{Misiats2016} for \eqref{eq:one_dim_fourier} it is not hard to prove that for any smooth, compactly supported function $f \in C^{\infty}_c(\mathbf{R}^2)$

\begin{align}
\sup_{(i, j) \in \epsilon_n\mathbf{Z}^2} \left| S_n(h, f)((i,j)) - G_h^{\mathbf{R}^2} * f((i, j))\right| = & \operatorname{Lip}(f) o(\epsilon_n) + \sup |f| O\left(\frac{\epsilon_n^2}{h}\right)\label{eq:concrete_est}
\\  & +\sup |f| O\left(\frac{\epsilon_n}{\sqrt{h}}\log\left( \frac{\epsilon_n}{\sqrt{h}}\right)\right),
\end{align}
where $G_h^{\mathbf{R}^2}$ denotes the heat kernel in the Euclidean plane at time $h$. In particular, when $\epsilon_n = h_n^{\alpha}$ for $\alpha \ge \frac{3}{2}$, we see that \eqref{eq:concrete_est} implies \eqref{eq:needed_heat_est}. This allows us to use Theorem~\ref{thm:conditional_conv} to recover the results of \cite{Misiats2016} when $\alpha \ge \frac{3}{2}$. Actually, an inspection of the proof of Theorem \ref{thm:conditional_conv} shows that to check that $u^*$ and $u_*$ are, respectively, a viscosity subsolution and a viscosity supersolution to mean curvature flow, the estimate \eqref{eq:needed_heat_est} can be replaced by 
\begin{equation}
\max_{ x \in V_n} \left| S(h_n, f)(x) - e^{-h\kappa\Delta_{\xi}}f(x)\right| = (\sup |f| )\,o(\sqrt{h_n}) + \operatorname{Lip}(f)\,O(h_n^{\gamma}),
\end{equation}
for some $\gamma > 1$. In particular, we see that in the setting of the two-dimensional regular grid this is satisfied whenever $\epsilon_n = h_n^\gamma$. This allows to recover the full parameter range $\gamma > 1$ of \cite{Misiats2016}. We need the slightly sharper assumption $\gamma = \frac{3}{2}$ for checking the initial conditions for $u^*$ and $u_*$.
\end{remark}

\subsection{Results on the MBO scheme and on the approximate MBO scheme}
The MBO scheme as stated in Algorithm~\ref{algo_exact} corresponds to the choices $S_n(t, \cdot) = e^{-t\Delta_n}(\cdot)$, the heat semigroup on the $n$-th graph, which acts on functions $u \in \mathcal{V}_n$ by
\begin{equation}
e^{-t\Delta_{n}}(u)(x) = \sum_{y \in V_n} H_n(t, x, y)u(y).
\end{equation}
Let $0 < K_n \le n$ be a sequence of numbers converging to $+\infty$, then the approximate MBO scheme as stated in Algorithm~\ref{algo_approx} corresponds to the choices $S_n = P_n$, where the operators $P_n$ act on functions $u \in \mathcal{V}_n$ by
\begin{equation}\label{eq:definition_approximate_op}
P_n(t, u)(x) := \sum_{y \in V_n} H_n^{K_n}(t, x, y)u(y).
\end{equation}
Our second main result states that on random geometric graphs the operators $e^{-t\Delta_n}(\cdot)$ and $P_n$ satisfy the assumptions of Theorem~\ref{thm:conditional_conv} with high probability.

\begin{theorem}\label{thm:estimate_random_graphs}
Let us assume that $\nu := \rho\volm$ is a probability measure with a smooth and positive density $\rho$. Assume that the points $\{x_i\}_{i=1}^{+\infty}$ in the above construction are i.i.d. random points sampled from $M$, distributed according to $\nu$. Assume that $q >0, \frac{2}{k} > s>0$ are such that:
\begin{enumerate}[(i)]
\item $q >\frac{1}{\frac{2}{k}-s}$,\label{eq:cond_sq}
\item We have that $\inf_{i \in \mathbf{N}} (\lambda_{i+1} - \lambda_{i}) > 0$.\label{eq:cond_spect_gap}
\item $K_n \ge (\operatorname{log}(n))^q$,\label{eq:K_n_def}
\item $h_n \gg (\log(n))^{-\alpha}$, with $\alpha = -1+\frac{2q}{k} - sq\ge 0$,\label{eq:h_n_cond}
\item $\epsilon_n \ll (\log(n))^{-\beta}$, with $\beta = -\frac{1}{2}+ 4q +\frac{13q}{k} -\frac{sq}{2} \ge 0$,\label{eq:cond_eps}
\item We have
\begin{equation}
\epsilon_n \gtrsim \begin{cases}
\left( \frac{\operatorname{log}(n)}{n}\right)^{\frac{1}{k}}\ &\text{if}\ k \ge 3,
\\ \left( \frac{\operatorname{log}(n)}{n}\right)^{\frac{1}{8}}\ &\text{if}\ k = 2,
\end{cases}
\end{equation}\label{eq:lower_bdd_eps}
\end{enumerate}
Then the operators $e^{-t\Delta_n}(\cdot)$ and $P_n$ satisfy conditions \eqref{item:monotonicity}, \eqref{item:estimate} and \eqref{item:approx_one} in Theorem~\ref{thm:conditional_conv}  (with $\xi = \rho^2$ and $\kappa = \kappa(\eta)$) on $G_n$ with probability greater than
\begin{equation*}
1 - C\epsilon_{n}^{-6k}\operatorname{exp}(-\frac{n\epsilon_n^{k+4}}{C}) - Cn\operatorname{exp}(-\frac{n}{C(\log(n))^{2q}}).
\end{equation*}
\end{theorem}

\begin{remark}\label{rem:remark_after_estimate}
Let us comment on this second result.
\begin{enumerate}[(i)]
\item For each $k \ge 2$, the space of admissible parameters $(s, q)$ in Theorem~\ref{thm:estimate_random_graphs} is quite \emph{large}. To see this, we plot the space of admissible parameters. The shaded region represents the space of admissible pairs $(s, q)$.
\begin{center}
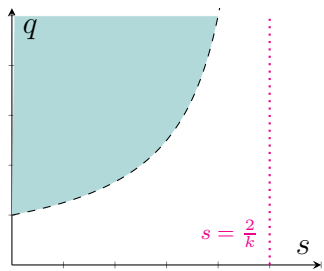

\begin{tikzpicture}[scale=0.6]
\begin{axis}[
    axis lines = middle,
    xlabel = {$s$},
    ylabel = {$q$},
        yticklabels={,,},
    xticklabels={,,},
    xmin=0, xmax=0.60,
    ymin=0, ymax=10.3]
 
\addplot [name path = A,
    domain = 0:0.49,
    samples = 1000,
    dashed] {1/(0.5-x)} 
    node [very near start, right] {\tiny $f(s) = \frac{1}{\frac{2}{k}-s}$};
 
\addplot [name path = B,
    domain = 0.005:0.9,
    color = white] {10} 
    node [pos=1, above] {};
 
\addplot [teal!30] fill between [of = A and B, soft clip={domain=0:0.4}];
 \addplot[thick, samples=50, smooth,domain=0:2,magenta, dotted] coordinates {(0.5,0)(0.5, 10)} node [very near start, left] {\tiny $s=\frac{2}{k}$};

\end{axis}
\end{tikzpicture}
\captionof{figure}{\small{Parameter space.}}\label{fig:third}
\end{center}
\medskip

\item Condition \eqref{eq:cond_spect_gap} in Theorem~\ref{thm:estimate_random_graphs} concerns the geometry of the manifold $M$. It implies in particular that the eigenvalues of the Laplacian $\Delta_{\rho^2}$ are simple. Condition \eqref{eq:cond_spect_gap} in Theorem~\ref{thm:estimate_random_graphs} is for example satisfied by the $k$-torus and by the $k$-sphere with standard unit density, see \cite[Chapter II, Section 2]{Chavel1984} and \cite[Chapter II, Section~4]{Chavel1984}.
\item Let us observe that conditions \eqref{eq:cond_eps} and \eqref{eq:lower_bdd_eps} in Theorem~\ref{thm:estimate_random_graphs} are compatible, indeed the right-hand side of \eqref{eq:cond_eps} in Theorem~\ref{thm:estimate_random_graphs} is a rational function of $\operatorname{log}(n)$, while the lower bound in condition \eqref{eq:lower_bdd_eps} in Theorem~\ref{thm:estimate_random_graphs} converges to zero as a power of $n$, up to a logarithmic factor. We also remark that items \eqref{eq:h_n_cond} and \eqref{eq:cond_eps} of Theorem~\ref{thm:estimate_random_graphs} imply
\begin{equation}
{\epsilon_n} \lesssim {h_n^{3/2}},
\end{equation}
while we expect that the convergence of the scheme should be true up to the critical scaling
\begin{equation}
{\epsilon_n} \ll {h_n}.
\end{equation}
Observe furthermore that condition \eqref{eq:h_n_cond} in Theorem \ref{thm:estimate_random_graphs} gives a lower bound for $h_n$ of the form
\begin{equation}
    h_n \gg ( \log(\delta_n) )^\alpha,
\end{equation}
where $\delta_n = (\frac{1}{n})^{1/k}$ is the characteristic distance between the nodes of the graph. This is perhaps not too surprising because the diffusion needs some time to smear out the fine details in the graph that appear at its characteristic length scale. A similar condition already appeared in \cite{Dunson2021}.
\item In the proof of Theorem \ref{thm:estimate_random_graphs} we will assume, for simplicity, that $K_n = \log(n)^q \in \mathbf{N}$. In this setting we will use condition \eqref{eq:cond_eps} of Theorem~\ref{thm:estimate_random_graphs} in the form
\begin{equation}\label{eq:how_we_use_it}
\epsilon_n \ll \frac{\sqrt{\log(n)}}{K_n^{1+\frac{1}{k} - \frac{s}{2}}\left( \lambda_{K_n}^{\frac{2}{k}+1} + 1\right)^2 \left(\lambda_{K_n}^{4+\frac{k}{2}} + 1 \right)}.
\end{equation}
Observe that condition \eqref{eq:cond_eps} of Theorem~\ref{thm:estimate_random_graphs} implies \eqref{eq:how_we_use_it} because by Weyl's law we have $\lambda_{K_n} \sim K_n^{2/k}$.
\end{enumerate}
\end{remark}

\begin{corollary}\label{corollary: conv_geom_graphs}
In the setting of Theorem~\ref{thm:estimate_random_graphs}, if we additionally assume that
\begin{equation}\label{eq:lower_bound_eps_as}
\epsilon_n \gg \left(\frac{\operatorname{log}(n)}{n}\right)^{\frac{1}{k+4}},
\end{equation}
then the conclusion of Theorem~\ref{thm:conditional_conv} holds almost surely both for the MBO scheme, Algorithm \ref{algo_exact}, and the approximate MBO scheme, Algorithm \ref{algo_approx}.
\end{corollary}

An important ingredient for the proof of Theorem \ref{thm:estimate_random_graphs} is the following lemma, which gives an estimate of the distance between the approximate heat kernel on the graph and the heat kernel on the manifold in a uniform sense. Such heat kernel estimates are of independent interest, for example, one should compare with \cite[Theorem 3]{Dunson2021}, where the authors obtain a similar estimate when the frequency cut-off $K_n$ and the time-scale $h_n$ are fixed. In Lemma \ref{lem:heat_kernel_estimate} we improve their result by showing how to choose $K_n$ in terms of $n$ as $n \to +\infty$.

\begin{lemma}\label{lem:heat_kernel_estimate}
In the setting of Theorem \ref{thm:estimate_random_graphs}, there exist constants $a_1, a_2, a_3, a_4 > 0$ such that if $n$ is large enough, then, with probability greater than $1 - a_1 \epsilon_n^{-6k}\operatorname{exp}(-a_2n \epsilon_n^{k+4}) \break - a_3 n\operatorname{exp}(-a_4\frac{n}{(\log(n))^{2q}})$, we have
\begin{equation}\label{eq:pointwise_comp}
\max_{x, y \in V_n}\left| H_{\epsilon_n}^{K_n}(h_n, x, y) - \frac{\rho(y)}{n}H(\kappa(\eta)h_n, x, y)\right| = o\left( \frac{\sqrt{h_n}}{n}\right).
\end{equation}
\end{lemma}

\section{The level set equation for MCF on a weighted manifold}\label{sec:mcf_drift}
In this section, we provide the basic framework for viscosity solutions to mean curvature flow in weighted Riemannian manifolds. 

Hereafter $(M, g)$ is a $k$-dimensional closed Riemannian manifold, and $\xi > 0$ is a smooth function on $M$. Recall that the evolution of a smooth open set $\Omega_0$ by mean curvature follows the trajectory of steepest descent of the area functional, which is defined as

\begin{equation}
\Omega \mapsto \int_{\partial \Omega} dS,
\end{equation}
where $\Omega$ ranges over all open sets in $M$ with a smooth boundary. When we consider a weight $\xi$ on the manifold, the correct functional to consider is the weighted-area functional, defined as

\begin{equation}
\Omega \mapsto \int_{\partial \Omega} \xi dS,
\end{equation}
where $\Omega$ ranges over all open sets in $M$ with smooth boundary. We define the evolution of mean curvature flow with density $\xi$ - hereafter denoted as MCF$_\xi$ - as the trajectory of steepest descent of this functional. To derive an equation for MCF$_\xi$ we consider  a family $\left\{ \Omega(t)\right\}_{0 \le t < T}$ of smooth open sets evolving smoothly in time with normal velocity vector $V$. Denote by $\nu(t)$ a suitable extension of the outer unit normal of $\partial \Omega(t)$. We then have by Gauss' Theorem

\begin{align*}
\frac{d}{dt}\int_{\partial \Omega(t)} \xi dS &= \frac{d}{dt} \int_{\Omega(t)} \frac{1}{\xi}\operatorname{div}(\xi \nu(t))\xi d\volm
\\ &=\int_{\partial \Omega(t)} \frac{1}{\xi}\operatorname{div}(\xi \nu(t))g(V(t), \nu(t)) \xi dS.
\end{align*}
We thus see that the trajectory of steepest descent is given by 

\begin{equation}
g(V, \nu) = -\frac{1}{\xi}\operatorname{div}(\xi \nu).
\end{equation}
We are thus led to the following definition.

\begin{definition}\label{def:mcf}
Let $(M, g)$ be a smooth $k$-dimensional closed Riemannian manifold. Let $\xi > 0$ be a smooth function on $M$. A family $\{\Omega_t\}_{t \ge 0}$ of smooth open subsets of $M$ is said to evolve by MCF$_{\xi}$  if
\begin{equation}\label{eq:eqmcf}
g(V, \nu) = -\frac{1}{\xi}\operatorname{div}(\xi \nu).
\end{equation}
where $V$ is the velocity vector field of the evolution and $\nu$ is the outer unit normal field.
\end{definition}
\begin{remark}
Using that the mean curvature $H(t)$ of $\partial \Omega(t)$ satisfies $\operatorname{div}(\nu(t)) = H(t)$, equation \eqref{eq:eqmcf} can be rewritten as
\begin{equation}\label{eq:mcfrewritten}
g(V, \nu) = -H - g\left(\frac{\nabla \xi}{\xi}, \nu \right),
\end{equation}
which yields the following interpretation for \eqref{eq:eqmcf}: the evolution by MCF$_\xi$ as defined in Definition~\ref{def:mcf} is driven by the minimization of two quantities, area and density. The first term on the right-hand side of \eqref{eq:mcfrewritten} forces the evolution to follow a trajectory which decreases as much as possible the area of $\partial \Omega(t)$, whereas the second term on the right-hand side forces the evolution to move towards regions where the density $\xi$ is low.
\end{remark}
We now derive the corresponding level set formulation for the above evolution in the spirit of \cite{Evans1991, Chen1991}. Let $u: [0, +\infty) \times M \to \mathbf{R}$ be a smooth function, assume for this heuristic discussion that $Du \neq 0$ everywhere. For any $s \in \mathbf{R}$ define $\Omega_t^s := \{ x \in M:\ u(t,x) > s\}$ and assume that $\{\Omega_t^s\}_{t \ge 0}$ evolves by $\text{MCF}_{\xi}$ defined in Definition~\ref{def:mcf}. Let $s \in \mathbf{R}$ and let $x: (0, T) \to M$ a smooth curve such that $x(t) \in \partial \Omega_t^s$ for every time $0 < t < T$. Then

\begin{align*}
0 & = \frac{d}{dt} u(t, x(t))
\\ & = (\partial_t u) (t, x(t)) + g(\nabla u(t, x(t)), \dot{x}(t)).
\end{align*}
Using the fact that the outer normal to the super level set $\Omega_t^s$ is given by $\nu(t, x) = -\frac{\nabla u(t, x)}{|\nabla u(t, x)|}$ and plugging in \eqref{eq:eqmcf} we obtain

\begin{align*}
(\partial_t u)(t, x(t)) &= |\nabla u(t, x(t))|g(\nu(t, x(t)),V(t, x(t))) 
\\ &=-|\nabla u(t, x(t))| \frac{1}{\xi(x(t))}\operatorname{div}(\xi \nu)(t,x(t)).
\end{align*}
Using the product rule for the divergence and recalling that $\nu = -\frac{\nabla u}{|\nabla u|}$ we observe that $u$ solves

\begin{equation}\label{eq:level_set_equation}
\partial_tu = \left\langle g - \frac{Du \otimes Du}{|Du|^2}, D^2u\right\rangle + g\left(\frac{\nabla \xi}{\xi}, \nabla u\right),
\end{equation}
where we denoted by $\langle \cdot, \cdot \rangle$ the extension of $g$ to the linear bundle of $T^*M \otimes T^*M$, i.e.\ for $A, B$ sections of $T^*M \otimes T^*M$ we have in local coordinates

\begin{equation}
\langle A, B \rangle := \sum_{i,j,k,l=1}^k A_{ij}g^{jk}g^{kl}B_{li}.
\end{equation}
From \eqref{eq:level_set_equation} we are led to the following definition.

\begin{definition}
Let $u: (0,T) \times M \to \mathbf{R}$ be a smooth function with $Du \neq 0$ everywhere. Then $u$ is said to solve the level set formulation of $\text{MCF}_{\xi}$ if \eqref{eq:level_set_equation} holds on $(0,T) \times M$.
\end{definition}

\begin{remark}
Another way of deriving directly equation \eqref{eq:level_set_equation} without relying on \eqref{eq:eqmcf} is by computing the steepest descent of the total variation functional $\int_M |\nabla u| \xi d\volm$ with respect to the metric
\begin{align*}
(\delta u, \delta u) = \int_M \left( \frac{\delta u}{|\nabla u|} \right)^2 |\nabla u|\xi d\volm.
\end{align*}
Indeed, consider a smooth function $u:(0, T) \times M \to \mathbf{R}$ with $Du \neq 0$, we then compute
\begin{align*}
\frac{d}{dt}\int_M |\nabla u(x, t)| \xi(x)d\volm &= \int_M g\left(\frac{\nabla u(x, t)}{|\nabla u(x, t)|}, \nabla \partial_t u(x, t)\right) \xi(x) d\volm
\\ & = -\int_M \operatorname{div}\left( \xi\frac{\nabla u}{|\nabla u|}\right)(t, x) \partial_t u(t, x)d\volm.
\end{align*}
Thus the steepest descent of the total variation functional with respect to the metric defined above is given by requiring
\begin{equation}
\partial_t u = |\nabla u| \frac{1}{\xi}\operatorname{div}\left( \xi \frac{\nabla u}{|\nabla u|}\right),
\end{equation}
which is equivalent to \eqref{eq:level_set_equation}.
\end{remark}
We are now ready to introduce a weak notion of solution for \eqref{eq:level_set_equation} based on the notion of viscosity solution. In the context of mean curvature flow with constant density $\xi = 1$ it was introduced in \cite{Evans1991} and \cite{Chen1991} in the Euclidean case, and in \cite{Ilmanen1992} on curved manifolds. If $U \subset (0,T) \times M$ is an open set, $(t_0, x_0) \in U$ and if $u: (0, T) \times M \to \mathbf{R}$ is an upper (lower) semi-continuous function, a smooth function $\varphi: U \to \mathbf{R}$ is said to be tangent to $u$ at $(t_0, x_0)$ from above (below), if $u-\varphi$ has a local maximum (minimum) at $(t_0, x_0)$.

\begin{definition}\label{def:visc_sol}
An upper (lower) semi-continuous function $u: (0, T) \times M \to \mathbf{R}$ is said to be a viscosity subsolution (supersolution) for \eqref{eq:level_set_equation} if for every $(t_0, x_0) \in (0, T) \times M$ and every smooth function $\varphi$ tanget to $u$ from above (below):
\begin{enumerate}[(i)]
\item If $D\varphi(t_0, x_0) \neq 0$ then
\begin{equation}
\partial_t\varphi \le \bigg\langle g - \frac{D\varphi \otimes D\varphi}{|D\varphi|^2},  D^2\varphi \bigg\rangle + g\left(\frac{\nabla \xi}{\xi}, \nabla \varphi\right)\  {(\ge)}\quad \text{at}\ (t_0, x_0)
\end{equation}\label{item:first_viscosity}
\item Otherwise there exists $\nu \in T_{x_0}^*M$ with $|\nu| \le 1$ such that
\begin{equation}
\partial_t\varphi \le \langle g - \nu \otimes \nu , D^2\varphi\rangle \  {(\ge)}\quad \text{at}\ (t_0, x_0)
\end{equation}\label{item:second_visc}
\end{enumerate}
We say that $u$ is a viscosity solution if it is both a subsolution and a supersolution.
\end{definition}
In \cite{Ilmanen1992} the author introduces the notion of viscosity subsolution/supersolution to mean curvature flow on a manifold (which corresponds to choosing the constant density $\xi = 1$) requiring continuity of the function $u$. We need to work with this slightly more general definition because the functions $u_*$ and $u^*$ in Theorem~\ref{thm:conditional_conv} are not continuous. We recall the following useful characterization of Definition~\ref{def:visc_sol}, which says that we need to check condition \eqref{item:second_visc} only when also $D^2\varphi(t_0,x_0) = 0$. 

\begin{proposition}\label{prop:equivalent_visc}
Let $u: (0, T) \times M \to \mathbf{R}$ be an upper (lower) semicontinuous function. Then $u$ is a viscosity subsolution (supersolution) of the level set formulation of $\text{MCF}_{\xi}$ if and only if whenever $\varphi$ is tangent to $u$ at $(t_0, x_0)$ from above (below), \eqref{item:first_viscosity} is satisfied and if $D\varphi(t_0, x_0) = 0$ and $D^2\varphi(t_0,x_0) = 0$, then
\begin{equation}
\partial_t \varphi(t_0, x_0) \le 0\ (\ge).
\end{equation}
\end{proposition}

Proposition~\ref{prop:equivalent_visc} is proved in the Euclidean case in \cite[Proposition 2.2]{Barles1995}. On a manifold, the proof is analogous. We recall the following comparison principle.
\begin{theorem}\label{thm:comparison}
Let $M$ be a closed $k$-dimensional Riemannian manifold. Let $\xi > 0$ be a smooth function on $M$. Let $u$ be a subsolution of \eqref{eq:level_set_equation} on $(0, T] \times M$ and let $v$ be a viscosity supersolution of \eqref{eq:level_set_equation} on $(0,T] \times M$. Define
\begin{equation}
u^*(x) := \limsup_{y \to x,\,t \to 0} u(t, y),\ v_*(x) := \liminf_{y \to x,\, t \to 0} v(t, y).
\end{equation}
Assume that $u^* \le v_*$ and that either $u^*$ or $v_*$ is continuous. Then for every $t \in (0, T]$
\begin{equation}
u(t, \cdot) \le v(t, \cdot).
\end{equation}
\end{theorem}
Theorem~\ref{thm:comparison} is proved when $\xi = 1$ is the constant density and the functions $u, v$ are assumed to be continuous in \cite{Ilmanen1992}. A careful look at the proof reveals that the same argument goes trough with the above assumptions. When $M = \mathbf{R}^k$ is the flat Euclidean space, an even more general version of Theorem~\ref{thm:comparison} can be found in \cite[Theorem 18]{Ambrosio2000}.
We also recall the following result concerning the existence of viscosity solutions, which can be again found in \cite{Ilmanen1992} for the case of a constant density $\xi = 1$.
\begin{theorem}\label{thm:existence_viscosity}
Let $M$ be a $k$-dimensional closed Riemannian manifold, and let $\xi>0$ be a smooth function on $M$. Let $u_0:M \to \mathbf{R}$ be continuous. Then there exists a unique viscosity solution $u:[0, T) \times M \to \mathbf{R}$ to \eqref{eq:level_set_equation} such that $u(0) = u_0$.
\end{theorem}

Finally, we recall the following \emph{relabeling property}, which is proved in \cite{Ilmanen1992} in the case of a constant density $\xi = 1$.
\begin{lemma}\label{lem:relab}
Let $M$ be a $k$-dimensional closed Riemannian manifold, and let $\xi > 0$ be a smooth function on $M$. Let $u: [0, T) \times M \to \mathbf{R}$ be a viscosity solution to \eqref{eq:level_set_equation}. Then for every continuous map $\Psi: \mathbf{R} \to \mathbf{R}$, the function $v := \Psi \circ u$ is a viscosity solution to \eqref{eq:level_set_equation}.
\end{lemma}

\section{MBO scheme on manifolds}\label{sec:mbo_manifolds}

As in the previous section, $M$ will denote a $k$-dimensional closed Riemannian manifold and $\xi > 0$ will denote a smooth function on $M$. The following algorithm can be used to approximate the evolution of an open set $\Omega_0 \subset M$ with smooth boundary by $\text{MCF}_{\xi}$.

\begin{algorithm}[MBO scheme on manifolds]\label{algo}
Fix a time-step size $h > 0$, a diffusion coefficient $\kappa > 0$ and a (bounded) drift $f:M \to \mathbf{R}$. Let $\Omega_0 \subset M$ be an open set with a smooth boundary. For each $n \in \mathbf{N}$ define inductively $\Omega_{l+1}$ as follows.
\begin{enumerate}
\item \textbf{Diffusion}. Define 
\begin{equation}
u_l := e^{-h\kappa\Delta_{\xi}}\mathbf{1}_{\Omega_l}.\nonumber
\end{equation}
\item \textbf{Thresholding}. Define $\Omega_{n+1}$ by
\begin{equation}
\Omega_{l+1} = \left\{ u_l \ge \frac{1}{2} + f\sqrt{h} \right\}.\nonumber
\end{equation}
\end{enumerate}
\end{algorithm}

We then have the following result for one step of MBO.
\begin{theorem}\label{thm:thresholding_one_step}
Let $M$, $\xi$ be as above. Let $\Omega_0$ be a smooth open set such that $\operatorname{diam}(\Omega_0) < \frac{\operatorname{inj}(M)}{2}$. Let $\Omega_1$ be obtained by applying one step of MBO with a bounded drift $f:M \to \mathbf{R}$ to $\Omega_0$ with a given step size $h>0$ and a given diffusion coefficient $\kappa > 0$. Let $x \in \partial \Omega_0$. Let $\nu(x) \in T_xM$ be the outer unit normal to $\partial \Omega_0$ at $x$ and define
\begin{equation}
z(x) := \begin{cases}
\sup\left\{s \in \mathbf{R}^{-}|\ \operatorname{exp}_x(s\nu(x)) \in \Omega_1 \right\}\ &\text{if}\ x \not \in \Omega_1,
\\
\,\inf\,\left\{s \in \mathbf{R}^{+}|\ \operatorname{exp}_x(s\nu(x)) \not \in \Omega_1 \right\}\ &\text{if}\ x \in \Omega_1.
\end{cases}
\end{equation}
Then we have
\begin{equation}
|z(x)| \le Vh,
\end{equation}
where the constant $V$ depends only on $\kappa$, the $L^\infty$-norm of $f$, the ambient manifold $M$, and the $C^0$-norm of the second fundamental form of $\partial \Omega_0$.
\end{theorem}
\begin{corollary}\label{corollary:thresholding_ball}
Let $x_0 \in M$ and $R < \frac{\operatorname{inj}(M)}{4}$ be fixed. Then there is a constant $C_R < +\infty$ such that if $\frac{R}{2} < r \le R$ and, in the above theorem, $\Omega_0 = B_r(x_0)$, then
\begin{equation}
|z(x)| \le C_Rh
\end{equation}
for every $x \in \partial B_r(x_0)$.
\end{corollary}
Finally, we have the following consistency result, which will be crucial in proving Theorem~\ref{thm:conditional_conv}.
\begin{theorem}\label{thm:consist}
Let $h_n$ be a sequence of positive real numbers converging to zero. Assume that $\psi_{h_n}: (0, +\infty) \times M \to \mathbf{R}$ are $C^{1,2}((0, +\infty) \times M)$ functions converging in $C^{1,2}((0, +\infty) \times M)$ to a function $\psi: (0, +\infty) \times M \to \mathbf{R}$. Assume that $(s_{h_n}, z_{h_n}) \in (0, +\infty) \times M$ are converging to a point $(s, z) \in [0, +\infty) \times M$. Assume also that $\delta_n := \psi_{h_n}(s_{h_n}, z_{h_n})$ are such that
\begin{align}\label{eq:delta_conv_rate}
\lim_{n \to +\infty} \frac{\delta_{n}}{\sqrt{h_n}} = 0.
\end{align}
Then we have that:
\begin{enumerate}[(i)]
\item If $D\psi(s, z) \neq 0$ then
\begin{align}
\liminf_{n \to +\infty} &\frac{1}{\sqrt{\kappa h_n}}\bigg(\frac{1}{2} - \int_{\left\{\psi_{h_n}(t_{h_n} - h_n, \cdot) \ge 0\right\}} H(\kappa h_n, z_{h_n}, y) \xi(y)d\volm\bigg)
\\ & \ge \frac{1}{2\sqrt{\pi}|D\psi(s, z)|}\bigg( \partial_t\psi - \bigg\langle g - \frac{D\psi \otimes D\psi}{|D\psi|^2} ,D^2\psi\bigg\rangle -  g\left( \frac{\nabla\xi}{\xi}, \nabla\psi\right)\bigg)(s, z).\label{eq:claim_non_deg}
\end{align}\label{item:nondeg_grad}
\item Otherwise if $D\psi(s,z) = 0, D^2\psi(s, z) = 0$ and
\begin{equation}\label{eq:claim_deg}
\frac{1}{2} - \int_{\left\{\psi_{h_n}(t_{h_n} - h_n, \cdot) \ge 0\right\}} H(\kappa h_n, z_{h_n}, y) \xi(y)d\volm \le o(\sqrt{h_n}),
\end{equation}
then 
\begin{equation}
\partial_t \psi(s,z) \le 0.
\end{equation}\label{item:deg_grad}
\end{enumerate}
\end{theorem} 
\section{Proofs}\label{sec:proofs}

\subsection{Conditional convergence: Proof of Theorem~\ref{thm:conditional_conv}}

The purpose of this section is the proof of Theorem~\ref{thm:conditional_conv}, which is inspired by the works \cite{Barles1995} and \cite{Misiats2016}.

\begin{proof}[Proof of Theorem~\ref{thm:conditional_conv}]
Let $u$ be the unique viscosity solution to MCF$_\xi$ from Theorem \ref{thm:existence_viscosity} with $\xi = \rho^2$, starting from $u(0, \cdot) = sd(\cdot, \Gamma_0) := d_M(x, \Omega_0^c) - d_M(x, \Omega_0)$. We will show later that $u^*$ and $u_*$ are, respectively, a viscosity subsolution and a viscosity supersolution of the level set formulation of $\text{MCF}_{\xi}$ according to Definition~\ref{def:visc_sol}. We furthermore claim that for every $x \in M$,
\begin{align}
&\begin{aligned}\label{eq:initial_cond:sub}
u^*(0, x) \le \operatorname{sign}^*(u(0,x)),
\end{aligned}
\\&\begin{aligned}\label{eq:initial_cond_sup}
u_*(0, x) \ge \operatorname{sign}_*(u(0,x)),
\end{aligned}
\end{align}
where $\operatorname{sign}^*$ and $\operatorname{sign}_*$ are, respectively, the upper semi-continuous envelope and the lower semi-continuous envelope of the sign function. 

Once these facts are proved, it follows from Theorem~\ref{thm:comparison} that for every $x \in M$ and every $t \ge 0$,
\begin{align}
&\begin{aligned}\label{eq:after_comp_sub}
u^*(t, x) \le \operatorname{sign}^*(u(t,x)),
\end{aligned}
\\&\begin{aligned}\label{eq:after_comp_sup}
u_*(t, x) \ge \operatorname{sign}_*(u(t,x)).
\end{aligned}
\end{align}
To see this, we observe that if $\Psi: \mathbf{R} \to \mathbf{R}$ is a continuous function such that $\Psi \ge \operatorname{sign}^*$, then the relabeling property in Lemma~\ref{lem:relab} implies that $\Psi \circ u$ is a continuous solution to \eqref{eq:level_set_equation} with $u^*(0, x) \le \operatorname{sign}^*(u(0,x)) \le \Psi(u(0,x))$ for every $x \in M$, thus Theorem~\ref{thm:comparison} implies that for every $0 \le t \le T$ and every $x \in M$
\begin{equation}
u^*(t, x) \le \inf_{\Psi \in C(\mathbf{R}), \Psi \ge \operatorname{sign}^*} \Psi(u(t,x)) = \operatorname{sign}^*(u(t, x)).
\end{equation} 
A similar argument gives \eqref{eq:after_comp_sup}. Let us now conclude the proof of the theorem assuming that \eqref{eq:after_comp_sub} and \eqref{eq:after_comp_sup} hold. If $x \in \Omega_t$, then $u(t, x) > 0$, thus \eqref{eq:after_comp_sup} yields $u_*(t,x) = 1$. In a similar way \eqref{eq:after_comp_sub} implies that $u^*(t,x) = -1$ on $(\Omega_t \cup \Gamma_t)^c$. We are thus left with proving that $u^*$ is a subsolution, that $u_*$ is a supersolution and with verifying the initial conditions \eqref{eq:initial_cond:sub} and \eqref{eq:initial_cond_sup}.

We now show that indeed $u^*$ is a viscosity subsolution. Pick a test functions $\varphi$ tangent to $u^*$ at $(t_0, x_0) \in (0, +\infty) \times M$ from above. We may assume without loss of generality that
\begin{equation}\label{eq:decay_in_time}
\lim_{t\to +\infty} \max_{M} \varphi(t, \cdot) = +\infty,
\end{equation}
and that $u^*-\varphi$ has a strict global maximum at $(t_0, x_0)$. Thanks to Proposition~\ref{prop:equivalent_visc}, we only need to check that
\begin{enumerate}
\item Either $D\varphi(t_0,x_0) \neq 0$ and 
\begin{equation}
\partial_t\varphi \le \bigg\langle g - \frac{D\varphi \otimes D\varphi}{|D\varphi|^2}, D^2\varphi\bigg\rangle + g\left(\frac{\nabla \xi}{\xi}, \nabla \varphi\right)\  \text{at}\ (t_0, x_0).
\end{equation}
\item Or $D\varphi(t_0, x_0) = 0$, $D^2\varphi(t_0,x_0) = 0$ and
\begin{equation}
\partial_t \varphi (t_0,x_0) \le 0.
\end{equation}
\end{enumerate}
If $(t_0, x_0) \in \{ u^* = -1 \}$ or $(t_0, x_0) \in \operatorname{Int}\{ u^* = 1 \}$ the claim is trivial, because in that case $u^*$ is constant in a neighborhood of $(t_0, x_0)$. We thus assume that $(t_0, x_0) \in \partial \{u^* = 1\}$. By definition, there exists a sequence $(t_{n_j}, z_{n_j})$ such that $z_{n_j} \in G_{n_j}$ for every $j \in \mathbf{N}$ and, as $j \to +\infty$,

\begin{align*}
n_j &\to +\infty,
\\ z_{n_j} &\to x_0,
\\ t_{n_j} &\to t_0,
\\ u^{n_j, G_{n_j}}(t_{n_j}, z_{n_j}) &\to u^*(t_0, x_0).
\end{align*}
For every $j \in \mathbf{N}$, pick

\begin{equation}\label{eq:choice_mac}
(s_j, x_j) \in \operatorname{argmax}_{x \in G_{n_j}, s \in (0, +\infty)} \left\{u^{n_j, G_{n_j}}(s, x)  - \varphi(s,x) \right\}.
\end{equation}
We observe that, up to extracting a subsequence, $(s_j, x_j) \to (t_0, x_0)$ as $j \to +\infty$. Indeed by the compactness of $M$ and the assumption \eqref{eq:decay_in_time}, we may assume that the sequence $(s_j, x_j)$ converges to some limit point $(\underline{s}, \underline{x})$. Then by definition of $u^*$, by the choice \eqref{eq:choice_mac} and by the properties of the points $(t_{n_j}, z_{n_j})$ we must have

\begin{align*}
(u^*-\varphi)(\underline{s}, \underline{x}) & \ge \limsup_{j \to +\infty} \ (u^{n_j, G_{n_j}} - \varphi)(s_j, x_j)
\\ &\ge \limsup_{j \to +\infty}\ (u^{n_j, G_{n_j}} - \varphi)(t_{n_j}, z_{n_j})
\\ &=(u^* - \varphi)(t_0, x_0).
\end{align*}
This forces $(t_0, x_0) = (\underline{s}, \underline{x})$, because $(t_0, x_0)$ is a strict global maximum for $u^* - \varphi$. It is also easy to check that $u^{n_j, G_{n_j}}(s_j, x_j) = 1$ for $j$ large enough. We now pick a sequence $\delta_j \downarrow 0$ to be determined later, and we define $\theta_{j}: \mathbf{R} \to [-1, 1]$ to be a smooth function such that

\begin{align*}
&\theta_j(t) = \operatorname{sign}(t)\ \text{for}\ |t| \ge \delta_{j},
\\ &\Vert \theta_j' \Vert_{\infty} \le \frac{2}{\delta_{j}}.
\end{align*}
We claim that
\begin{equation}\label{eq:first_claim}
u^{n_j, G_{n_j}}(s, z) \le \theta_j(\varphi(s, z) - \varphi(s_j, x_j) + \delta_j)
\end{equation}
for every $j$ large enough, $z \in G_{n_j}$ and $s \in (0, +\infty)$. Indeed, inequality \eqref{eq:first_claim} holds trivially if $u^{n_j, G_{n_j}}(s, z) = -1$. If instead $u^{n_j, G_{n_j}}(s, z) = 1$, probing \eqref{eq:choice_mac} with $(s, z)$, we have
\begin{align*}
1 &= u^{n_j, G_{n_j}}(s, z) \le u^{n_j, G_{n_j}}(s_j, x_j) - \varphi(s_j, x_j) + \varphi(s, z)
 \\ &= 1 - \varphi(s_j, x_j) + \varphi(s, z),
\end{align*}
where we used that $u^{n_j, G_{n_j}}(s_j, x_j) = 1$ for $j$ large enough. In particular
\begin{equation}
0 \le - \varphi(s_j, x_j) + \varphi(s, z),
\end{equation}
which, by definition of $\theta_j$, yields \eqref{eq:first_claim}. 

We now choose $s = s_j - h_{n_j}$ in \eqref{eq:first_claim}, we apply $ S_{n_j}(h_{n_j}, \cdot)$ to both sides of the inequality and we evaluate at $x_j$. Recalling assumption \eqref{item:monotonicity} of Theorem~\ref{thm:conditional_conv} we get
\begin{align*}
& S_{n_j}(h_{n_j},u^{n_j, G_{n_j}}(s_j - h_{n_{j}}, \cdot))(x_j)
\\ & \le S_{n_j}(h_{n_j}, \theta_j(\varphi(s_j - h_{n_j}, \cdot) - \varphi(s_j, x_j) + \delta_j))(x_j) + O\left(h_{n_j}^{3/2}\right).
\end{align*}
We now apply $\operatorname{sign}^*$ to both sides of the inequality to get
\begin{align*}
1 = u^{n_j, G_{n_j}}(s_j, x_j) \le \operatorname{sign}^*\left( S_{n_j}(h_{n_j}, \theta_j(\varphi(s_j - h_{n_j}, \cdot) - \varphi(s_j, x_j) + \delta_j))(x_j) + O\left(h_{n_j}^{3/2}\right) \right),
\end{align*}
which, by definition of the function $\operatorname{sign}^*$, implies
\begin{align*}
0 \le S_{n_j}(h_{n_j}, \theta_j(\varphi(s_j - h_{n_j}, \cdot) - \varphi(s_j, x_j) + \delta_j))(x_j) + O\left(h_{n_j}^{3/2}\right).
\end{align*}
We now divide both sides of the previous inequality by $2$ and we add $1/2$ to both sides of the inequality. Using assumption \eqref{item:approx_one} of Theorem~\ref{thm:conditional_conv} and the linearity of $S_n$ in the second variable yields
\begin{align*}
\frac{1}{2} \le S_{n_j}\left(h_{n_j}, \left( \frac{1 + \theta_j}{2}\right)\bigg(\varphi(s_j - h_{n_j}, \cdot) - \varphi(s_j, x_j) + \delta_j\bigg)\right)(x_j) + O\left(h_{n_j}^{3/2}\right).
\end{align*}
Define
\begin{align*}
f_j(z) := \left( \frac{1 + \theta_j}{2}\right)\bigg(\varphi(s_j - h_{n_j}, z) - \varphi(s_j, x_j) + \delta_j\bigg).
\end{align*}
Then by applying the estimate \eqref{eq:needed_heat_est} in assumption \eqref{item:estimate} in Theorem~\ref{thm:conditional_conv} we obtain
\begin{equation}
\frac{1}{2} \le (e^{-h_{n_j}\kappa\Delta_{\xi}}f_j)(x_j) + o(h_{n_j}^{1/2}) + \frac{2}{\delta_j}O(h_{n_j}^{3/2}).
\end{equation}
In other words, we have
\begin{align*}
o\left(h_{n_j}^{1/2}\right) + \frac{2}{\delta_j}O\left(h_{n_j}^{3/2}\right) & \ge \frac{1}{2} - \int_M H(h_{n_j}\kappa, x_j, y)f_j(y)\xi(y)d\volm(y)
\\ & \ge \frac{1}{2} - \int_{\{\varphi(s_j - h_{n_j}, \cdot) - \varphi(s_j, x_j) + \delta_j \ge 0\}} H(h_{n_j}\kappa, x_j, y)\xi(y)d\volm(y).
\end{align*}
We divide the previous inequality by $\sqrt{h_{n_j}\kappa}$, and we choose $\delta_j = h_{n_j}^{2/3}$ so that on the one hand $\frac{h_{n_j}}{\delta_j} \to 0$ and on the other hand we can apply Theorem~\ref{thm:consist}. If $D\varphi(t_0, x_0) \neq 0$, then by \eqref{item:nondeg_grad} in Theorem~\ref{thm:consist},
\begin{align*}
0 \ge  \frac{1}{2\sqrt{\pi}|D\psi(s, z)|}\bigg( \partial_t\psi - \bigg\langle g - \frac{D\psi \otimes D\psi}{|D\psi|^2}, D^2\psi\bigg\rangle -  g\left( \frac{D\xi}{\xi}, D\psi\right)\bigg)(t_0, x_0),
\end{align*}
which gives \eqref{item:first_viscosity} in Definition~\ref{def:visc_sol}. If $D\varphi(t_0, x_0) = 0$ and $D^2\varphi(t_0, x_0) = 0$ then we can apply \eqref{item:deg_grad} in Theorem~\ref{thm:consist} to get the second item in the equivalent description of viscosity subsolution in Proposition~\ref{prop:equivalent_visc}. Thus $u^*$ is a viscosity subsolution. In a similar way one can prove that $u_*$ is a supersolution. 

We are left with checking the initial conditions for $u^*$ and $u_*$. Again, we focus on the inequality \eqref{eq:initial_cond:sub} for $u^*$, since the argument for $u_*$ is similar. Observe that
\begin{equation}
\operatorname{sign}^*(u(0,x)) = \begin{cases}
1\ &\text{if}\ x \in \overline{\Omega_0}
\\
-1\ &\text{if}\ x \in M\setminus \overline{\Omega_0}
\end{cases}
\end{equation}
and since $u^* \in \{-1, 1\}$, we just have to show that $u^*(0,x) = -1$ for $x \in M \setminus \overline{\Omega_0}$. To this aim, pick a sequence $(t_n, z_n) \in (0,+\infty) \times G_n$ such that $t_n \to 0$ and $z_n \to x$ as $n \to +\infty$. We have to show that $u^{n, G_n}(t_n, z_n) = -1$ for $n$ large enough. For $q \in \mathbf{R}$, denote by $T^{q, G_n}(h_n)(\Omega_0)$ the outcome of the abstract thresholding scheme with thresholding value given by $q$ and step size $h_n$ on the graph $G_n$ with initial value $\Omega_0 \cap V_n$. For $m \in \mathbf{N}$ we also write $(T^{q, G_n}(h_n))^m$ for $T^{q, G_n}(h_n) \circ . . .  \circ T^{q, G_n}(h_n)$. Since $x \in M \setminus \overline{\Omega_0}$ there exists $R > 0$ such that $B_{R}(x) \subset M \setminus \overline{\Omega_0}$. We denote by $w_n: V_n \to [0, +\infty)$ a sequence of non-negative functions which, for $n$ large enough and for every $u, v \in \mathcal{V}_n, |u| \le 1, |v| \le 1$, satisfy
\begin{equation}\label{eq:also_this}
u \le v \Rightarrow S(h_n, u) \le S(h_n, v) + w_n,
\end{equation}
\begin{equation}
a_n := \Vert w_n \Vert_{L^\infty(G_n)} = O(h_n^{3/2}),
\end{equation}
\begin{equation}\label{eq:this_we_need}
\max_{x \in V_n}|S(h_n, \mathbf{1}_{G_n})(x) - 1| < a_n.
\end{equation}
Such functions exist by assumptions \eqref{item:monotonicity} and \eqref{item:approx_one} in Theorem~\ref{thm:conditional_conv}. We now check that
\begin{equation}\label{eq:inclusion}
V_n\setminus \left(T^{1/2, G_n}(h_n)\right)^m({\Omega_0}) \supset \left(T^{1/2 + 2ma_n, G_n}(h_n)\right)^m(B_R(x)).
\end{equation}
To see this, we proceed by induction over $m$. We treat just the base case $m = 1$, the inductive step being analogous. To prove \eqref{eq:inclusion} for $m=1$, we show 
\begin{equation}\label{eq:theprovedinclusion}
V_n \setminus T^{1/2, G_n}(h_n)(\Omega_0) \supset T^{1/2 +a_n, G_n}(h_n)(M\setminus \Omega_0) \supset T^{1/2 + 2a_n, G_n}(h_n)(B_R(x)).
\end{equation}
To see this, let $y \in T^{1/2+ a_n, G_n}(h_n)(M\setminus \Omega_0)$, observe that by \eqref{eq:this_we_need} we have

\begin{equation}
S(h_n, \mathbf{1}_{\Omega_0})(y) +  \frac{1}{2} + a_n \le S(h_n, \mathbf{1}_{\Omega_0})(y) + S_n(h_n, \mathbf{1}_{M\setminus \Omega_0})(y) < 1 + a_n,
\end{equation}
in particular, we have that $y \in V_n \setminus T^{1/2, G_n}(h_n)(\Omega_0)$. Thus $ V_n \setminus T^{1/2, G_n}(h_n)(\Omega_0) \supset T^{1/2 + a_n, G_n}(M\setminus \Omega_0)$. We now observe that since $\mathbf{1}_{B_R(x)} \le \mathbf{1}_{M \setminus \Omega_0}$, \eqref{eq:also_this} yields that for $y \in T^{1/2+2a_n, G_n}(h_n)(B_R(x))$
\begin{equation}
\frac{1}{2} + 2a_n \le S(h_n, \mathbf{1}_{B_R(x)})(y) \le S(h_n, \mathbf{1}_{M\setminus \Omega_0})(y) +a_n,
\end{equation}
which yields \eqref{eq:theprovedinclusion}.

We will show that there is a constant $C < +\infty$ such that 
\begin{equation}\label{eq:claim_for_ic}
\left(T^{1/2 + 2\left[ \frac{t_n}{h_n}\right]a_n, G_n}(h_n)\right)^{\left[\frac{t_n}{h_n}\right]}(B_R(x)) \supset B_{R-Ct_n}(x) \cap V_n.
\end{equation}
Once this is proved, we have that using also \eqref{eq:inclusion}, since $t_n \downarrow 0$, 
\begin{equation}
M\setminus \left(T^{1/2, G_n}(h_n)\right)^{\left[ \frac{t_n}{h_n} \right]}({\Omega_0}) \supset B_{\frac{R}{2}}(x)
\end{equation}
when $n$ is large enough. In particular, since $z_n$ is converging to $x$, we must have that $u^{n, G_n}(t_n, z_n) = -1$ for $n$ large enough. Finally, to show \eqref{eq:claim_for_ic} we argue as follows. Let $C_R$ be the constant in Corollary \ref{corollary:thresholding_ball}. Let $f \in C^{\infty}_c(B_{R}(x))$ such that $\mathbf{1}_{B_{R-C_R h_n}(x)} \le f \le \mathbf{1}_{B_{R}(x)}$ with $\operatorname{Lip}(f) \le c/h_n$, using assumptions \eqref{item:monotonicity} and \eqref{item:estimate} in Theorem~\ref{thm:conditional_conv} we have for $y \in M \cap V_n$
\begin{align*}
S_n(h_n, \mathbf{1}_{B_R(x)})(y) & \ge S_n(h_n, f)(y) + O(h_n^{3/2})
\\ & \ge e^{-h_n\kappa\Delta_\xi}f(y) + O(h_n^{1/2})
\\ & \ge e^{-h_n\kappa\Delta_\xi}\mathbf{1}_{B_{R-C_Rh_n}(x) }(y)+ O(h_n^{1/2}).
\end{align*}
Observe that $\frac{1}{2} + 2\left[ \frac{t_n}{h_n} \right] a_n = \frac{1}{2} + O(h_n^{1/2})$, in particular, we can apply Corollary \ref{corollary:thresholding_ball} to obtain, for $n$ large enough, whenever $y \in B_{R-2C_Rh_n}(x) \cap V_n$
\begin{equation}
e^{-h_n\kappa\Delta_\xi}\mathbf{1}_{B_{R-C_Rh_n}(x) }(y)+ O(h_n^{1/2}) \ge\frac{1}{2} + 2\left[ \frac{t_n}{h_n} \right] a_n.
\end{equation}
By an induction argument we get \eqref{eq:claim_for_ic}.
\end{proof}

\subsection{Heat kernel estimate in random geometric graphs: Proof of Theorem~\ref{thm:estimate_random_graphs}}
The main purpose of this subsection is the proof of Theorem~\ref{thm:estimate_random_graphs}. We first introduce some notation. We denote by $\{\lambda_l\}_{l=1}^{+\infty}$ the eigenvalues of the weighted Laplacian $\Delta_{\rho^2}$ on the manifold $(M, g)$, which are ordered in the following way (recall that we are assuming that the eigenvalues are simple)
\begin{equation}
0 = \lambda_1 < \lambda_2 < \lambda_3 < . . . 
\end{equation}
We denote by $\{f_l\}_{l=1}^{+\infty}$ an orthonormal basis (with respect to the $L^2(\rho^2\volm)$-inner product on $M$) made of the corresponding eigenvectors. Then, for $x, y \in M$, the heat kernel on $M$ can be written as

\begin{equation}\label{eq:manifold_hk_expansion}
H(t, x, y) = \sum_{l=1}^{+\infty} e^{-t\lambda_l}f_l(x)f_l(y).
\end{equation}

\begin{proof}[Proof of Theorem~\ref{thm:estimate_random_graphs}]
As we pointed out in Remark~\ref{rem:remark_after_estimate}, in the present proof we will for simplicity assume that $K_n = \log(n)^q \in \mathbf{N}$. We will indicate by $\gamma$ the quantity $\gamma:= \inf_{i \in \mathbf{N}}(\lambda_{i+1}-\lambda_i)$, which is positive by Item \eqref{eq:cond_spect_gap} in Theorem~\ref{thm:estimate_random_graphs}.

Observe that items \eqref{item:monotonicity} and \eqref{item:approx_one} in Theorem~\ref{thm:conditional_conv} hold exactly (i.e.\ without error) for the choice $S_n(t, \cdot) = e^{-t\Delta_n}(\cdot)$. To show that these hold true with high probability also for the choice $S_n = P_n$ defined in \eqref{eq:definition_approximate_op} we take $w \in \mathcal{V}_n$ and we consider, for $x \in V_n$, the difference
\begin{align*}
\bigg|e^{-h_n\Delta_n}w(x) - P_n(h_n, w)(x)\bigg| &= \left| \sum_{y \in V_n} \sum_{l=K_n +1}^ne^{-h_n\lambda_n^l}v_n^l(x)v_n^l(y)\frac{d_n(y)}{n}w(y) \right|
\\ & \le n \max_{z \in V_n} |w(z)|  \max_{z \in V_n} |d_n(z)|\frac{1}{n}\max_{z \in V_n}\sum_{l =K_n + 1}^n e^{-h_n \lambda_n^l}(v_n^l(z))^2,
\end{align*}
where in the last line we used the Cauchy--Schwarz inequality. To get items \eqref{item:monotonicity} and  \eqref{item:approx_one} in Theorem~\ref{thm:estimate_random_graphs} for $P_n$, it thus suffices to show that 
\begin{equation}
\mathcal{R}_n :=  \max_{z \in V_n} d_n(z) \max_{z \in V_n}\frac{1}{n}\sum_{l =K_n + 1}^n e^{-h_n \lambda_n^l}(v_n^l(z))^2 = O\left( \frac{h_n^{3/2}}{n} \right).
\end{equation}
To show this, we start by observing that for every $n \in \mathbf{N}$, every $z \in V_n$ and $1 \le l \le n$
\begin{equation}\label{eq:pointwise_bound_ev_graph}
1 = \langle v_n^l, v_n^l \rangle_{\mathcal{V}_n} \ge \frac{d_{n}(z)}{n}(v_n^l(z))^2.
\end{equation}
By applying Theorem~\ref{thm:density_est} we can also choose $n$ so large that, with probability greater than $1-Q_6\epsilon_n^{-k}\operatorname{exp}(-Q_7n\epsilon_n^{k+2})$, we have
\begin{equation}
 \max_{z \in V_n} |d_{n}(z) -C_1\rho(z)| \le Q_8\epsilon_n,
\end{equation}
and we can clearly assume that $n$ is so large that
\begin{equation}
C_1\frac{\min \rho}{2} \le d_n \le 2C_1\max \rho.
\end{equation}
Using \eqref{eq:pointwise_bound_ev_graph} and the ordering $\lambda_n^l \ge \lambda_n^{K_n}$ for $n  \ge l \ge K_n$ we get
\begin{align*}
\mathcal{R}_n &\le \frac{C}{n}\left( n^2 e^{-\lambda_n^{K_n}h_n}\right)
\\ &= \frac{C}{n}\left( n^2 e^{-\kappa(\eta)\lambda_{K_n} h_n}e^{-\left(\lambda_n^{K_n} - \kappa(\eta)\lambda_{K_n}\right)h_n}\right).
\end{align*}
We now use Theorem~\ref{thm:eigv_eigf_est} and Theorem~\ref{thm:controlling_norm_ef} to infer that with probability greater than $1 - Q_1 \epsilon_n^{-6k}\operatorname{exp}(-Q_2n\epsilon_n^{k+4}) - Q_3n\operatorname{exp}(-Q_4n\left( \lambda_{K_n} + 1 \right)^{-k})$ we have
\begin{align*}
\mathcal{R}_n &\le \frac{C}{n}\left( n^2 e^{-\kappa(\eta)\lambda_{K_n}h_n}e^{\frac{C\epsilon_n}{\gamma}\left(\lambda_{K_n}^{4+\frac{k}{2}} + 1 \right)h_n}\right).
\end{align*}
By Weyl's law we have that $\lambda_{K_n} \sim K_n^{2/k}$, thus
\begin{align*}
\mathcal{R}_n \le \frac{C}{n}\left( n^2 e^{-cK_n^{2/k}h_n}e^{\frac{\tilde{C}\epsilon_n}{\gamma}K_n^{\frac{8}{k}+1}}\right).
\end{align*}
Recalling the conditions \eqref{eq:h_n_cond}, \eqref{eq:cond_eps} and \eqref{eq:cond_spect_gap} in Theorem~\ref{thm:estimate_random_graphs}, as well as the scaling $K_n = (\log(n))^q$ we get
\begin{align*}
\mathcal{R}_n &\le \frac{C}{n}\left(n^2e^{-c(\log(n))^{\frac{2q}{k} - \alpha}} \right) 
\\ &= \frac{Ch_n^{3/2}}{n}\left( \frac{n^{2-c(\log(n))^{\frac{2q}{k}-1-\alpha}}}{h_n^{3/2}}\right)
\\ &\le \frac{Ch_n^{3/2}}{n}\left( n^{2-c(\log(n))^{\frac{2q}{k}-1-\alpha}} (\log(n))^{\frac{3\alpha}{2}}\right).
\end{align*}
So $\mathcal{R}_n = O\left( \frac{h_n^{3/2}}{n}\right)$ because by the definition of $\alpha$ in \eqref{eq:h_n_cond} in Theorem \ref{thm:estimate_random_graphs} we have $\frac{2q}{k}-1-\alpha > 0$.
\\
We are left with proving item \eqref{item:estimate} in Theorem~\ref{thm:conditional_conv} for both $e^{-t\Delta_n}(\cdot)$ and $P_n$. We prove it for $e^{-t\Delta_n}(\cdot)$, the proof for $P_n$ being analogous. The proof is divided into three steps.

Step 1. We claim that with probability greater than $1 - a_1 \epsilon_n^{-6k}\operatorname{exp}(-a_2n \epsilon_n^{k+4}) - a_3 n\operatorname{exp}(-a_4n\left( \lambda_{K_n} + 1 \right)^{-k})$
\begin{equation}\label{eq:pointwise_comp_HK}
\max_{x, y \in V_n}\left| H_{\epsilon_n}^{n}(h_n, x, y) - \frac{\rho(y)}{n}H(\kappa(\eta)h_n, x, y)\right| = o\left( \frac{\sqrt{h_n}}{n}\right).
\end{equation}
To show \eqref{eq:pointwise_comp_HK} we pick two points $x, y \in V_n$ and we compute
\begin{align*}
    \bigg| H_{\epsilon_n}^{n}(h_n, x, y) - \frac{\rho(y)}{n}H(\kappa(\eta)h_n, x, y)\bigg| \le &\bigg| H_{\epsilon_n}^{K_n}(h_n, x, y) - \frac{\rho(y)}{n}H(\kappa(\eta)h_n, x, y)\bigg|
    \\ &+ \bigg| \sum_{l=K_n+1}^{n} e^{-h_n\lambda_n^l} v_n^l(x) v_n^l(y) \frac{d_n(y)}{n}\bigg|.
\end{align*}
From Lemma \ref{lem:heat_kernel_estimate} we get that the first term on the right-hand side is $o\left( \frac{\sqrt{h_n}}{n}\right)$ with probability greater than $1 - a_1 \epsilon_n^{-6k}\operatorname{exp}(-a_2n \epsilon_n^{k+4}) - a_3 n\operatorname{exp}(-a_4n\left( \lambda_{K_n} + 1 \right)^{-k})$, while the second term is estimated in the same way as the term $\mathcal{R}_n$ in the previous part of the proof.

Step 2. We choose an optimal transport map
\begin{equation}
T_n \in \underset{T_{\#} \nu = \nu_n}{\operatorname{argmin}} \sup_{x \in M} d_M(x, T(x)),\ \theta_n := \sup_{x \in M} d_M(x, T_n(x)).
\end{equation}
We claim that, with probability greater than $1 - a_1 \epsilon_n^{-6k}\operatorname{exp}(-a_2n \epsilon_n^{k+4}) \break - a_3 n\operatorname{exp}(-a_4n\left( \lambda_{K_n} + 1 \right)^{-k})$, we have for every $f \in C^{\infty}(M)$,
\begin{align}
\max_{x \in V_n} \left| e^{-h_n\Delta_n}f (x) - e^{-\kappa(\eta)h_n\Delta_{\rho^2}}f(x)\right| \le &L_1 \sup_M |f| \frac{\theta_n}{\sqrt{h_n}}e^{\frac{2\theta_n \operatorname{diam}(M)}{h_n}}.\nonumber
\\ &+ \sup_M |f| o(\sqrt{h_n}) + L_2 \left( \sup_{M}|f|+ \operatorname{Lip}(f)\right) \theta_n,\label{eq:est_to_use_for_conclusion}
\end{align}
where the constants $L_1, L_2$ and the function in $o(\sqrt{h_n})$ depend only on $M$.

To show \eqref{eq:est_to_use_for_conclusion}, we work under the assumption that we are in the event in which the estimate of Step 1 holds true; this happens with probability greater than
\begin{equation}
1 - a_1 \epsilon_n^{-6k}\operatorname{exp}(-a_2n \epsilon_n^{k+4}) - a_3 n\operatorname{exp}(-a_4n\left( \lambda_{K_n} + 1 \right)^{-k}).
\end{equation}
We take $f \in C^{\infty}(M)$ and $x \in V_n$. Then by using the triangle inequality
\begin{align*}
&\begin{aligned}
|e^{-h_n\Delta_n}f(x) - e^{-\kappa(\eta)h_n\Delta_{\rho^2}}f(x)| 
\end{aligned}
\\ &\begin{aligned}
= \left|\sum_{y \in V_n} H_{\epsilon_n}^n(h_n, x, y)f(y) - \int_M H(\kappa(\eta)h_n, x, y)f(y) \rho^2(y)d\volm(y) \right|
\end{aligned}
\\ &\begin{aligned}
\le &\ \sum_{y \in V_n} \bigg|H_{\epsilon_n}^n(h_n, x, y)f(y) - \frac{\rho(y)}{n}H(\kappa(\eta)h_n, x, y)f(y)\bigg| 
\\ &+\bigg|\sum_{y \in V_n}\frac{\rho(y)}{n}H(\kappa(\eta)h_n, x, y)f(y) - \int_MH(\kappa(\eta)h_n, x, y)f(y)\rho^2(y)d\volm(y)\bigg|.
\end{aligned}
\end{align*}
For the first term on the right-hand side, we use the estimate in Step 1 to infer
\begin{align*}
\sum_{y \in V_n} \bigg|H_{\epsilon_n}^n(h_n, x, y)f(y) - \frac{\rho(y)}{n}H(\kappa(\eta)h_n, x, y)f(y)\bigg| &\le n \sup_M |f| o\left( \frac{\sqrt{h_n}}{n}\right)
\\ &= \sup_M |f| o(\sqrt{h_n}).
\end{align*}
For the second term, we recall that $(T_n)_{\#}\nu = \nu_n$, thus
\begin{align*}
&\bigg|\sum_{y \in V_n} \frac{\rho(y)}{n}H(\kappa(\eta)h_n, x, y)f(y) - \int_MH(\kappa(\eta)h_n, x, y)f(y)\rho^2(y)d\volm(y)\bigg|
\\ &= \bigg|\int_MH(\kappa(\eta)h_n, x, T_n(y))f(T_n(y))\rho(T_n(y))d\nu(y) - \int_MH(\kappa(\eta)h_n, x, y)f(y)\rho(y)d\nu(y)\bigg|.
\end{align*}
By the smoothness of $\rho$ and $f$, we observe that
\begin{align*}
\bigg|\int_M H(\kappa(\eta)h_n, x, y)\left( f(T_n(y))\rho(T_n(y)) - f(y)\rho(y)\right)d\nu(y)\bigg| \le L_2\left( \sup_M |f| + \operatorname{Lip}(f)\right) \theta_n,
\end{align*}
so we are left with showing that
\begin{align}\label{eq:last_estimate_to_prove}
&\left| \int_M (H(h_n, x, T_n(y)) - H(h_n, x, y))f(T_n(y))\rho(T_n(y))d\nu(y) \right|\nonumber
\\ &\le L_1 \sup_M |f| \frac{\theta_n}{\sqrt{h_n}}e^{\frac{\theta_n \operatorname{diam}(M)}{h_n}}.
\end{align}
To prove \eqref{eq:last_estimate_to_prove} we fix $x, y \in M$ and we consider the length minimizing constant-speed geodesic $\sigma_{n, y}: [0,1] \to M$ from $y$ to $T_n(y)$, i.e.,
\begin{equation}
\operatorname{Length}(\sigma_{n, y}|_{[0,s]}) = d_M(y, \sigma_{n, y}(s)).
\end{equation}
By the fundamental theorem of calculus, the Cauchy--Schwarz inequality and the boundedness of $\rho$ we obtain
\begin{align}
&\left| \int_M (H(h_n, x, T_n(y)) - H(h_n, x, y))f(T_n(y))\rho(T_n(y))d\nu(y) \right| 
\\ &\le C\sup_M |f| \int_0^1\int_M |\nabla H(h_n, x, \sigma_{n, y}(s))||\dot{\sigma}_{n,y}(s)|d\nu(y)ds\nonumber
\\ &\le C\theta_n \sup_M |f| \int_0^1 \int_M \frac{\hat{Q}_1}{\sqrt{h_n} \mu(B_{\sqrt{h_n}}(x))} \operatorname{exp}\left( -\frac{d_M^2(x,\sigma_{n, y}(s))}{\hat{Q}_2 h_n}\right) d\nu(y)ds,\label{eq:before_est_d}
\end{align}
where in the last line we used the fact that the speed of the constant-speed geodesic $\sigma_{n, y}$ is equal to its length -- which can be bounded by $C\theta_n$ by definition of $\theta_n$ -- and we estimated the gradient of the heat kernel by an application of Theorem~\ref{thm:gaussian_bounds}. We now observe that by the reverse triangle inequality
\begin{align*}
|d_M^2(x, \sigma_{n, y}(s)) - d_M^2(x,y)| &= (d_M(x, y)-d_M(x, \sigma_{n, y}(s)))(d_M(x, \sigma_{n, y}(s)) + d_M(x, y))
\\ &\le 2\theta_nd_M(x, y).
\end{align*}
Inserting this estimate into \eqref{eq:before_est_d} and using the Gaussian lower bound for the heat kernel from Theorem~\ref{thm:gaussian_bounds} yields
\begin{align*}
&\left| \int_M (H(h_n, x, T_n(y)) - H(h_n, x, y))f(T_n(y))\rho(T_n(y))d\nu \right| 
\\ &\le C\frac{\theta_n}{\sqrt{h_n}}e^{\frac{2\theta_n \operatorname{diam}(M)}{h_n}}\sup_M |f| \int_M H(\tilde{Q}h_n, x, y)d\nu(y)
\\ &\le L_1 \sup_M |f| \frac{\theta_n}{\sqrt{h_n}}e^{\frac{2\theta_n \operatorname{diam}(M)}{h_n}}.
\end{align*}

Step 2. Conclusion. To conclude the proof of the theorem from \eqref{eq:est_to_use_for_conclusion} one clearly just needs to prove that
\begin{equation}
\limsup_{n \to +\infty} \frac{\theta_n}{h_n^{3/2}} < +\infty.
\end{equation}
We first treat the case $k \ge 3$. Observe that, by Theorem~\ref{thm:transp_plan_existence}
\begin{equation}
\limsup_{n\to +\infty} \frac{n^{1/k}\theta_n}{\operatorname{log}^{1/k}(n)} < +\infty.
\end{equation}
In particular, using also assumption \eqref{eq:lower_bdd_eps}
\begin{align*}
\limsup_{n \to +\infty}\frac{\theta_n}{h_n^{3/2}} &= \limsup_{n \to +\infty}\left(\frac{n^{1/k}\theta_n }{\operatorname{log}^{1/k}(n)}\frac{\operatorname{log}^{1/k}(n)}{\epsilon_n n^{1/k}}\frac{\epsilon_n}{h_n^{3/2}}\right) < +\infty,
\end{align*}
provided
\begin{equation}\label{eq:eps_h_n}
\limsup_{n \to +\infty} \frac{\epsilon_n}{h_n^{3/2}} < +\infty.
\end{equation}
To check that \eqref{eq:eps_h_n} is satisfied, we observe that by the assumptions \eqref{eq:h_n_cond} and \eqref{eq:cond_eps} in Theorem~\ref{thm:estimate_random_graphs} we get
\begin{align*}
\limsup_{n \to +\infty} \frac{\epsilon_n}{h_n^{3/2}} \le \limsup_{n \to +\infty}\ (\log(n))^{\frac{3}{2}\alpha - \beta},
\end{align*}
the right-hand side of which is finite since assumption \eqref{eq:cond_sq} in Theorem \ref{thm:estimate_random_graphs} implies $\frac{3}{2}\alpha - \beta \le 0$.
\\
For the case $k = 2$ we proceed analogously. Recall that by Theorem~\ref{thm:transp_plan_existence}
\begin{equation}
\limsup_{n \to +\infty} \frac{n^{1/2}\theta_n}{\operatorname{log}^{3/4}(n)}< +\infty.
\end{equation}
In particular, using also assumption \eqref{eq:lower_bdd_eps} in Theorem \ref{thm:estimate_random_graphs} we obtain
\begin{equation}
\limsup_{n \to +\infty}\frac{\theta_n}{h_n^{3/2}} = \limsup_{n \to +\infty}\left(\frac{\theta_n n^{1/2}}{\operatorname{log}^{3/4}(n)}\left(\frac{\operatorname{log}(n)}{\epsilon_n^8 n}\right)^{1/2}\frac{\epsilon_n^4\operatorname{log}^{1/4}(n)}{h_n^{3/2}}\right) < +\infty,
\end{equation}
provided
\begin{equation}
\limsup_{n\to +\infty} \frac{\epsilon_n^4\operatorname{log}^{1/4}(n)}{h_n^{3/2}} < +\infty.
\end{equation}
To show this, we estimate $\epsilon_n$ using assumption \eqref{eq:cond_eps} in Theorem~\ref{thm:estimate_random_graphs} and esimate $h_n$ using assumption \eqref{eq:h_n_cond} in Theorem \ref{thm:estimate_random_graphs}
\begin{align*}
\limsup_{n \to +\infty}\frac{\epsilon_n^4\operatorname{log}^{1/4}(n)}{h_n^{3/2}} &\le\limsup_{n \to +\infty}\ (\log(n))^{\frac{1}{4}+\frac{3}{2}\alpha - 4 \beta} <+\infty,
\end{align*}
which follows from \eqref{eq:cond_sq} in Theorem \ref{thm:estimate_random_graphs}.
\end{proof}

\begin{proof}[Proof of Lemma \ref{lem:heat_kernel_estimate}]
As in the proof of Theorem \ref{thm:estimate_random_graphs}, we will for simplicity assume that $K_n = \log(n)^q \in \mathbf{N}$. We will indicate by $\gamma$ the quantity $\gamma:= \inf_{i \in \mathbf{N}}(\lambda_{i+1}-\lambda_i)$, which is positive by Item \eqref{eq:cond_spect_gap} in Theorem~\ref{thm:estimate_random_graphs}.

To show \eqref{eq:pointwise_comp}, fix two points $x, y \in V_n$. By using the expansion \eqref{eq:manifold_hk_expansion} and the triangle inequality we have
\begin{equation}
\left| H_{\epsilon_n}^{K_n}(h_n, x, y) - \frac{\rho(y)}{n}H(\kappa(\eta)h_n, x, y)\right| \le \mathbf{I}_n + \mathbf{II}_n,
\end{equation}
where we define
\begin{align*}
\mathbf{I}_n &= \left|\sum_{l=1}^{K_n - 1}  e^{-h_n\lambda_n^l}v_n^l(x)v_n^l(y)\frac{d_n(y)}{n} -  e^{-h_n\kappa(\eta)\lambda^l}f_l(x)f_l(y)\frac{\rho(y)}{n}\right|,
\\ \mathbf{II}_n &=\left| \sum_{l=K_n}^{+\infty} e^{-h_n\kappa(\eta)\lambda^l}f_l(x)f_l(y)\frac{\rho(y)}{n}\right|.
\end{align*}
We now proceed to show that these two terms are both of order $o\left(\frac{\sqrt{h_n}}{n}\right)$.
\\
To control term $\mathbf{II}_n$ we follow the ideas in \cite{Dunson2021} and \cite{Berard1994}. By the Cauchy--Schwarz inequality and by the fact that $\rho$ is bounded we get
\begin{equation}
\mathbf{II}_n \le \frac{C}{n}\max_{z \in M} \sum_{l=K_n}^{+\infty} e^{-h_n \kappa(\eta)\lambda_l} f_l^2(z).
\end{equation}
To control the right hand side, fix $z \in M$. We define a measure $\omega_z$ on $\mathbf{R}$ by
\begin{equation}
\omega_z := \sum_{l = K_n}^{+\infty} f_l^2(z) \delta_{\lambda_l}(d\lambda).
\end{equation}
Then an integration by parts yields
\begin{align*}
&\begin{aligned}
\sum_{l=K_n}^{+\infty} e^{-h_n\kappa(\eta)\lambda_l} f_l^2(z) &=\int_{\mathbf{R}} e^{-\kappa(\eta)h_n\lambda}d\omega_z(d\lambda) 
\end{aligned}
\\ &\begin{aligned}
\phantom{\sum_{l=K_n}^{+\infty} e^{-h_n\kappa(\eta)\lambda_l} f_l^2(z)} = \left[  e^{-\kappa(\eta)h_n\lambda}\omega_z([0, \lambda])\right]_{\lambda = 0}^{+\infty} + \int_{\mathbf{R}} \kappa(\eta)h_n e^{-\kappa(\eta)h_n\lambda}\omega_z([0, \lambda])d\lambda
\end{aligned}
\\ &\begin{aligned}
\phantom{\sum_{l=K_n}^{+\infty} e^{-h_n\kappa(\eta)\lambda_l} f_l^2(z)} \le &\limsup_{\lambda \to +\infty}  \left(e^{-h_n\kappa(\eta)\lambda}\sum_{\lambda_{K_n} \le \lambda_l \le \lambda} f_l^2(z) \right) 
\\ &+ \int_{\lambda_{K_n}}^{+\infty}h_n\kappa(\eta) e^{-h_n\kappa(\eta)\lambda}\omega_z([0, \lambda])d\lambda.
\end{aligned}
\end{align*}
Now we use Theorem~\ref{thm:gaussian_bounds} to show that the first term on the right hand side vanishes. Recalling the notation $\mu := \xi \volm$, and using the Gaussian upper bounds in Theorem~\ref{thm:gaussian_bounds} we get in particular
\begin{align}
\sum_{\lambda_{K_n} \le \lambda_l \le \lambda} f_l^2(z) \le e \sum_{0 \le \lambda_l \le \lambda} e^{-\frac{\lambda_l}{\lambda}} f_l^2(z) &\le eH\left( \frac{1}{\lambda}, z, z\right)\label{eq:abound}
\\ &\le \frac{C}{\mu (B_{\lambda^{-1/2}}(x) ))} \le C \lambda^{\frac{k}{2}},\nonumber
\end{align}
so that indeed 
\begin{align*}
\limsup_{\lambda \to +\infty} e^{-h_n \frac{\kappa(\eta)}{2}\lambda} \sum_{\lambda_{K_n} \le \lambda_l \le \lambda} f_l^2(z) \le \limsup_{\lambda \to +\infty} e^{-h_n\frac{\kappa(\eta)}{2}\lambda}C\lambda^{\frac{k}{2}} = 0.
\end{align*}
We thus obtain, using \eqref{eq:abound} once more with $\lambda_{K_n}$ replaced by zero,
\begin{align*}
\mathbf{II}_n &\le \frac{C}{n} \int_{\lambda_{K_n}}^{+\infty} h_n\kappa(\eta)e^{-h_n\kappa(\eta)\lambda}\lambda^{k/2}d\lambda
\\ &= \frac{C}{n} \left( h_n\kappa(\eta)\right)^{-\frac{k}{2}}\int_{\kappa(\eta)h_n\lambda_{K_n}}^{+\infty} e^{-\lambda}\lambda^{k/2}d\lambda
\\ & \le \frac{C}{n} h_n^{-\frac{k}{2}}\int_{ch_nK_n^{2/k}}^{+\infty} e^{-\lambda}\lambda^{k/2}d\lambda,
\end{align*}
where we used Weyl's law in the last step. If $ch_nK_n^{\frac{2}{k}} - \frac{k}{2} \ge 1$, we can estimate the right hand side by
\begin{align*}
\frac{C}{n} h_n^{-\frac{k}{2}}\left(ch_nK_n^{\frac{2}{k}}\right)^{\frac{k}{2}+1} e^{-ch_nK_n^{\frac{2}{k}}} &= \frac{\tilde{C}}{n} K_n e^{-A}A,
\end{align*}
where $A = ch_nK_n^{\frac{2}{k}}$.
Now we  follow the reasoning in the proof of \cite[Theorem 3]{Dunson2021} to obtain $K_nAe^{-A} \le \frac{1}{K_n}e^{-\frac{A}{2}}$ provided $A \ge 8\operatorname{log}(K_n)$, which is satisfied because of our assumption \eqref{eq:h_n_cond} in Theorem~\ref{thm:estimate_random_graphs}. Thus, using again our assumptions on $h_n$
\begin{align*}
\mathbf{II}_n &\le \frac{\tilde{C}\sqrt{h_n}}{n}\left( \frac{e^{-c(\log(n))^{\frac{2q}{k} -\alpha}}}{(\log(n))^q\sqrt{h_n}} \right)
\\ &\le \frac{\tilde{C}\sqrt{h_n}}{n}\left( e^{-c(\log(n))^{\frac{2q}{k}-\alpha}}(\log(n))^{\frac{\alpha}{2}-q} \right).
\end{align*}
Thus we obtain that $\mathbf{II}_n = o\left( \frac{\sqrt{h_n}}{n} \right)$ because of the definition of $\alpha$.

Regarding the term $\mathbf{I}_n$, we use the triangle inequality, to decompose this into four terms
\begin{equation}
\mathbf{I}_n \le \mathbf{I}_n^a + \mathbf{I}_n^b + \mathbf{I}_n^c + \mathbf{I}_n^d,
\end{equation}
where
\begin{align*}
\mathbf{I}_n^a &= \left| \sum_{l=1}^{K_n -1 } \left( e^{-h_n\lambda_n^l} - e^{-\kappa(\eta)h_n\lambda_l}\right)\frac{\rho(y)}{n}f_l(x)f_l(y)\right|,
\\ \mathbf{I}_n^b &= \left| \sum_{l=1}^{K_n -1 } e^{-h_n\lambda_n^l}\left(C_1 \frac{\rho(y)}{n} - \frac{d_n(y)}{n}\right)\frac{f_l(x)}{C_1^{1/2}}\frac{f_l(y)}{C_1^{1/2}} \right|,
\\ \mathbf{I}_n^c &= \left| \sum_{l=1}^{K_n -1 } e^{-h_n\lambda_n^l}\frac{d_n(y)}{n}\left(\frac{f_l(x)}{C_1^{1/2}} - v_n^l(x)\right)\frac{f_l(y)}{C_1^{1/2}} \right|,
\\ \mathbf{I}_n^d &= \left| \sum_{l=1}^{K_n -1 } e^{-h_n\lambda_n^l}\frac{d_n(y)}{n}v_n^l(x)\left(\frac{f_l(y)}{C_1^{1/2}} - v_n^l(y)\right) \right|.
\end{align*}
We now proceed at estimating these four terms.
\\
Term $\mathbf{I}_n^a$. We observe that $\lambda_n^1 = \lambda_1 = 0$, thus in the sum we can neglect the term corresponding to $l=1$, i.e.
\begin{equation}\label{eq:first_of_II}
\mathbf{I}_n^a \le \frac{C}{n}\sum_{l=2}^{K_n -1 } \left| e^{-h_n\lambda_n^l} - e^{-h_n\kappa(\eta)\lambda_l}\right|\Vert f_l \Vert_{C^0(M)}^2.
\end{equation}
Since $s \mapsto e^{-s}$ is $1$-Lipschitz continuous on $[0, +\infty)$, for every $2 \le l \le K_n -1 $ we have
\begin{align*}
\left| e^{-h_n\lambda_n^l} - e^{-\kappa(\eta)h_n\lambda_l}\right| \le |\lambda_n^l-\kappa(\eta)\lambda_l|h_n \le Q_5 \frac{\Vert f_l \Vert_{C^3(M)}}{\gamma}\epsilon_nh_n,
\end{align*}
where the last inequality holds with probability greater than $1 - Q_1 \epsilon_n^{-6k}\operatorname{exp}(-Q_2n\epsilon_n^{k+4}) \break - Q_3n\operatorname{exp}(-Q_4n\left( \lambda_{\overline{l}} + 1 \right)^{-k})$ because of Theorem~\ref{thm:eigv_eigf_est}. In particular using also Theorem~\ref{thm:controlling_norm_ef} to control the $C^0$ and $C^3$ norm of the eigenfunctions and using the fact that for $l \le K_n$ we have $\lambda_l \le \lambda_{K_n}$ we can bound
\begin{equation}
\mathbf{I}_n^a \le \frac{Ch_n}{n}\left( \frac{K_n{\left(\lambda_{K_n}^{1 + \frac{k}{2}} + 1 \right)^2\left( \lambda_{K_n}^{4 + \frac{k}{2}} + 1\right)\epsilon_n}}{\gamma}\right).
\end{equation}
From this, we obtain that $\mathbf{I}_n^a = o\left( \frac{\sqrt{h_n}}{n}\right)$, because by our assumptions on $\epsilon_n$  in \eqref{eq:cond_eps} of Theorem~\ref{thm:estimate_random_graphs} and our assumptions on the spectral gap in \eqref{eq:cond_spect_gap} of Theorem~\ref{thm:estimate_random_graphs} we clearly have
\begin{equation}
\left(\frac{K_n\left(\lambda_{K_n}^{1 + \frac{k}{2}} + 1 \right)^2\left( \lambda_{K_n}^{4 + \frac{k}{2}} + 1\right)\epsilon_n}{\gamma} \right) = o(1).
\end{equation}
\\
Term $\mathbf{I}_n^b$. Using Theorem~\ref{thm:eigv_eigf_est}, Theorem~\ref{thm:density_est} and Theorem~\ref{thm:controlling_norm_ef} we have that with probability greater than $1 - Q_1 \epsilon_n^{-6k}\operatorname{exp}(-Q_2n\epsilon_n^{k+4}) - Q_3n\operatorname{exp}(-Q_4n\left( \lambda_{\overline{l}} + 1 \right)^{-k}) \break - Q_6\epsilon_n^{-k}\operatorname{exp}(-Q_7n \epsilon_n^{k+2}),$ for each $1 \le l \le K_n -1$ we can estimate
\begin{align*}
& \left|e^{-h_n\lambda_n^l}\left( C_1 \frac{\rho(y)}{n} - \frac{d_n(y)}{n}\right) \frac{f_l(x)}{C_1^{1/2}}\frac{f_l(y)}{C_1^{1/2}} \right|
\\ & \le \frac{C}{n}e^{-h_n\kappa(\eta)\lambda_l}e^{-h_n\left( \lambda_n^l - \kappa(\eta)\lambda_l\right)}\Vert C_1\rho - d_n \Vert_{L^{\infty}(G_n)} \Vert f_l \Vert_{L^{\infty}(M)}^2
\\ & \le \frac{C}{n} e^{Ch_n \frac{\left( \lambda_{K_n}^{4+\frac{k}{2}} + 1 \right)\epsilon_n}{\gamma}}\left( \lambda_{K_n}^{1+\frac{k}{2}} + 1\right)^2\epsilon_n.
\end{align*}
In particular, multiplying and dividing by $\sqrt{h_n}$ and summing over $l = 1, . . . , K_n$, we obtain
\begin{equation}
\mathbf{I}_n^b \le  \frac{C\sqrt{h_n}}{n}\left( \frac{K_n}{\sqrt{h_n}}e^{ch_n \frac{\left( \lambda_{K_n}^{4+\frac{k}{2}} + 1 \right)\epsilon_n}{\gamma}}\left( \lambda_{K_n}^{1+\frac{k}{2}} + 1\right)^2\epsilon_n \right).
\end{equation}
By Weyl's law and our by assumptions \eqref{eq:cond_eps}, \eqref{eq:h_n_cond} and \eqref{eq:cond_spect_gap} in Theorem~\ref{thm:estimate_random_graphs}, this is again an $o\left( \frac{\sqrt{h_n}}{n}\right)$ term.
\\
The terms $\mathbf{I}_n^c, \mathbf{I}_n^d$ are treated similarly. In particular $\mathbf{I}_n= o\left( \frac{\sqrt{h_n}}{n}\right)$ provided we are in the event in which Theorem~\ref{thm:eigv_eigf_est} and Theorem~\ref{thm:density_est} apply. This happens with probability greater than 
\begin{align*}
&1 - Q_1 \epsilon_n^{-6k}\operatorname{exp}(-Q_2n\epsilon_n^{k+4}) - Q_3n\operatorname{exp}(-Q_4n\left( \lambda_{\overline{l}} + 1 \right)^{-k}) - Q_6\epsilon_n^{-k}\operatorname{exp}(-Q_7n \epsilon_n^{k+2})
\\ & \ge 1 - (Q_1+Q_6) \epsilon_n^{-6k}\operatorname{exp}(-\min(Q_2, Q_7)n\epsilon_n^{k+4}) - Q_3n\operatorname{exp}(-Q_4n\left( \lambda_{\overline{l}} + 1 \right)^{-k})
\\ &= 1 - a_1 \epsilon_n^{-6k}\operatorname{exp}(-a_2n \epsilon_n^{k+4}) - a_3 n\operatorname{exp}(-a_4n\left( \lambda_{K_n} + 1 \right)^{-k}),
\end{align*}
provided $n$ is large enough, this concludes our argument for \eqref{eq:pointwise_comp}.
\end{proof}

\begin{proof}[Proof of Corollary \ref{corollary: conv_geom_graphs}]
We know from Theorem~\ref{thm:estimate_random_graphs} that for $n$ large enough, assumptions \eqref{item:monotonicity}, \eqref{item:estimate}, \eqref{item:approx_one} of Theorem~\ref{thm:conditional_conv} hold for both the choices of the operators $e^{-t\Delta_n}$ and $P_n$ on the graph $G_n$ on an event $A_n$ such that
\begin{equation}
\mathbb{P}(A_n) \ge 1 - C\epsilon_{n}^{-6k}\operatorname{exp}(-\frac{1}{C}n\epsilon_n^{k+4}) - Cn\operatorname{exp}(-\frac{n}{C(\log(n))^{2q}}).
\end{equation}
For $\overline{n} \in \mathbf{N}$ large enough we consider the set
\begin{equation}
C_{\overline{n}} := \bigcap_{n \ge \overline{n}} A_n.
\end{equation}
Theorem~\ref{thm:conditional_conv} says that, in the event $C_{\overline{n}}$, for both the choices of the operators $e^{-t\Delta_n}$ and $P_n$ we have that \eqref{eq:lowersol} and \eqref{eq:uppersol} hold true. Observe that
\begin{align*}
\mathbb{P}(C_{\overline{n}}) \ge 1 - \sum_{n \ge \overline{n}} C\epsilon_{n}^{-6k}\operatorname{exp}(-\frac{1}{C}n\epsilon_n^{k+4}) - Cn\operatorname{exp}(-\frac{n}{C(\log(n))^{2q}}),
\end{align*}
In particular, we have that
\begin{align}
\mathbb{P}\left( \left\{ u^*\ \text{and}\ u_*\ \text{satisfy \eqref{eq:lowersol} and \eqref{eq:uppersol}}\right\}\right) &\ge \mathbb{P}\left( \bigcup_{\overline{n} \in \mathbf{N}} C_{\overline{n}} \right) = \lim_{\overline{n} \to +\infty} \mathbb{P}(C_{\overline{n}}) \nonumber
\\ &\begin{aligned}\ge 1 - \lim_{\overline{n} \to +\infty}\sum_{n \ge \overline{n}} \bigg(&C\epsilon_{n}^{-6k}\operatorname{exp}(-\frac{1}{C}n\epsilon_n^{k+4}) 
\\ &- Cn\operatorname{exp}(-\frac{n}{C(\log(n))^{2q}})\bigg).\label{eq:appear_series}
\end{aligned}
\end{align}
We thus  just need to show that
\begin{equation}
\lim_{\overline{n} \to +\infty}\sum_{n \ge \overline{n}} \bigg(C\epsilon_{n}^{-6k}\operatorname{exp}(-\frac{1}{C}n\epsilon_n^{k+4}) - Cn\operatorname{exp}(-\frac{n}{C(\log(n))^{2q}})\bigg) = 0,
\end{equation}
in other words, we need to prove that the series is convergent. To this end, observe that
\begin{align*}
C\epsilon_{n}^{-6k}\operatorname{exp}(-\frac{1}{C}n\epsilon_n^{k+4}) &= C\operatorname{exp}\left( -6k\operatorname{log}(\epsilon_n) -\frac{1}{C}n\epsilon_n^{k+4} \right)
\\ &= C\operatorname{exp}\left(\operatorname{log}(n)\left( -6k\frac{\operatorname{log}(\epsilon_n)}{\operatorname{log}(n)} -\frac{1}{C}\frac{n\epsilon_n^{k+4}}{\operatorname{log}(n)} \right)\right)
\\ &= Cn^{\left( -6k\frac{\operatorname{log}(\epsilon_n)}{\operatorname{log}(n)} -\frac{1}{C}\frac{n\epsilon_n^{k+4}}{\operatorname{log}(n)} \right)}.
\end{align*}
In a similar way, we have
\begin{equation}
Cn\operatorname{exp}(-\frac{n}{C(\log(n))^{2q}}) = Cn^{\left( 1 - \frac{1}{C}\frac{n}{(\log(n))^{2q+1}} \right)}.
\end{equation}
To prove the convergence of the series appearing in \eqref{eq:appear_series} it is sufficient to show
\begin{equation}
\lim_{n \to +\infty}\left( -6k\frac{\operatorname{log}(\epsilon_n)}{\operatorname{log}(n)} -\frac{1}{C}\frac{n\epsilon_n^{k+4}}{\operatorname{log}(n)} \right) = \lim_{n \to +\infty}\left( 1 - \frac{1}{C}\frac{n}{(\log(n))^{2q+1}} \right) = -\infty.
\end{equation}
The second limit is easily treated. To treat the first limit, observe that by assumption \eqref{eq:lower_bound_eps_as} in Corollary \ref{corollary: conv_geom_graphs} we have
\begin{equation}\label{eq:what_eps}
\lim_{n \to +\infty} \frac{n\epsilon_n^{k+4}}{\operatorname{log}(n)} = +\infty.
\end{equation}
To conclude the proof, we show that 
\begin{equation}\label{eq:claim_on_second}
\inf_{n \in \mathbf{N}}\frac{\operatorname{log}(\epsilon_n)}{\operatorname{log}(n)} > -\infty.
\end{equation}
Indeed, we have
\begin{align*}
\frac{\operatorname{log}(\epsilon_n)}{\operatorname{log}(n)} &= \frac{\operatorname{log}\left( \frac{\epsilon_n n^{\frac{1}{k+4}}}{\operatorname{log}^{\frac{1}{k+4}}(n)}\right)}{\operatorname{log}(n)} - \frac{1}{k+4} + \frac{\operatorname{log}\operatorname{log}(n)}{\operatorname{log}(n)},
\end{align*}
The first term is bounded from below because it is asymptotically nonnegative by \eqref{eq:lower_bound_eps_as} . The last term converges to zero as $n \to +\infty$. Thus \eqref{eq:claim_on_second} holds and the proof is complete.
\end{proof}

\subsection{MBO on manifolds}

\begin{proof}[Proof of Theorem~\ref{thm:thresholding_one_step}]
We let $\hat{x} := \operatorname{exp}_x(z(x)\nu(x))$. Then we have

\begin{equation}\label{eq:def_x_hat}
\frac{1}{2} + \omega_1\sqrt{h} = \int_{\Omega_0} H(\kappa h, \hat{x}, y)\rho^2(y)d\volm
\end{equation}
By the Gaussian upper bounds on the heat kernel in Theorem~\ref{thm:gaussian_bounds}, we have that $d_M(\hat{x}, \partial\Omega_0) \le \tilde{C}\sqrt{h}$, for a fixed constant $\tilde{C}$, independent of $\Omega_0$. In particular, we infer from the asymptotic expansion of the heat kernel in Theorem~\ref{thm:asymptotic_exp} that

\begin{equation}\label{eq:asexp}
\frac{1}{2} + \omega_1\sqrt{h} = \int_{\Omega_0} \frac{e^{-\frac{d_M^2(\hat{x},y)}{4\kappa h}}}{(4\pi \kappa h)^{k/2}} v_0(\hat{x},y)\rho^2(y) d\volm + O(h).
\end{equation}
Since $d(\hat{x}, \partial\Omega_0) \le \tilde{C}h$, and $\operatorname{diam}(\Omega_0) \le \frac{\operatorname{inj}(M)}{2}$, we can rewrite the integral in  \eqref{eq:asexp} in exponential coordinates around $\hat{x}$, i.e.

\begin{equation}\label{eq:asexp1}
\frac{1}{2} + \omega_1\circ \operatorname{exp}_{\hat{x}}\sqrt{h} = \int_{\tilde{\Omega}_0} \frac{e^{-\frac{|y|^2}{4\kappa h}}}{(4\pi \kappa h)^{k/2}} v_0(\hat{x},\operatorname{exp}_{\hat{x}}(y))\rho^2(\operatorname{exp}_{\hat{x}}(y)) dy + O(h),
\end{equation}
where $\tilde{\Omega}_0 := \operatorname{exp}_{\hat{x}}^{-1}(\Omega_0)$. Recalling that $v_0(\hat{x}, \hat{x}) = \frac{1}{\rho^2(\hat{x})}$, a Taylor expansion of the function $y \mapsto v_0(\hat{x}, \operatorname{exp}_{\hat{x}}(y)) \rho^2(\operatorname{exp}_{\hat{x}}(y))$ around zero reveals that 

\begin{equation}\label{eq:asexp_taylor}
\frac{1}{2} + \omega_1\circ \operatorname{exp}_{\hat{x}}\sqrt{h} = \int_{\tilde{\Omega}_0} \frac{e^{-\frac{|y|^2}{4\kappa h}}}{(4\pi \kappa h)^{k/2}} dy + O(\sqrt{h}).
\end{equation}
In other words, there exists a bounded function $\omega_2$ on $\mathbf{R}^k$ such that

\begin{equation}
\frac{1}{2} + \omega_2\sqrt{h} = \int_{\tilde{\Omega}_0} \frac{e^{-\frac{|y|^2}{4\kappa h}}}{(4\pi \kappa h)^{k/2}} dy.
\end{equation}
In other words, we have that $0 \in \partial E$, where

\begin{equation}
E = \left\{ v \in \mathbf{R}^k|\ \frac{1}{2} + \omega_2(v)\sqrt{h} \le \int_{\tilde{\Omega}_0} \frac{e^{-\frac{|v-y|^2}{4\kappa h}}}{(4\pi \kappa h)^{k/2}} dy\right\},
\end{equation}
and thus the normal distance $z(x)$ coincides with the normal distance of $\partial \tilde{\Omega}_0$ and $E$ at the point $\exp_{\hat{x}}^{-1}(x) \in \partial \tilde{\Omega}_0$. The conclusion of the proof is then obtained by applying the following result.
\end{proof}

\begin{proposition}\label{prop:thresh_one_step_euclid}
Let $\Omega \subset \mathbf{R}^k$ be a smooth open set. Let $E$ be obtained by applying one step of MBO with diffusion coefficient $\kappa > 0$, bounded drift $\omega:\mathbf{R}^k \to \mathbf{R}$ and step size $h>0$. Let $x \in \partial \Omega$. Let $\nu(x)$ be the outer unit normal to $\partial \Omega$ at $x$, define
\begin{equation}
z(x) := \begin{cases}
\sup\left\{l \in \mathbf{R}^{-}|\ x+l\nu(x) \in E \right\}\ &\text{if}\ x \not \in E
\\
\inf\left\{l \in \mathbf{R}^{+}|\ x+l\nu(x) \not \in E \right\}\ &\text{if}\ x \in E.
\end{cases}
\end{equation}
Then we have
\begin{equation}
|z(x)| \le \tilde{C}h,
\end{equation}
where the constant $\tilde{C}$ depends only on $k, \kappa$ and the $C^0$-norm of the second fundamental form of $\partial \Omega$.
\end{proposition}

Proposition~\ref{prop:thresh_one_step_euclid} is a weaker version of \cite[Theorem 4.1]{Fuchs2022}, which makes rigorous the original ideas in \cite{Mascarenhas}. In those works, the authors identify the exact first order coefficient of the expansion of $z(x)$ in $h$. Since we do not need this, we present a proof of our weaker statement.

\begin{proof}[Proof of Proposition~\ref{prop:thresh_one_step_euclid}]
For ease of notation, we assume that $\kappa = 1$. We treat the case when $z(x) > 0$, the other case is similar. First of all, we observe that $z(x) \le \tilde{C}_k\sqrt{h}$, for a constant $\tilde{C}_k$ depending just on the dimension $k$. We now choose a coordinate system in which $x = 0$ and $\nu(x) = e_k$. We may find an open set $U$ containing the origin and a smooth function $\zeta: \mathbf{R}^{k-1} \to \mathbf{R}$ such that $\zeta(0) = 0, D\zeta(0) = 0$ and
\begin{equation}
U \cap \Omega = \left\{ v \in \mathbf{R}^k|\ v_k < \zeta(v_1, . . . ,v_{k-1}) \right\}.
\end{equation}
Using the fact that $z(x) = O(\sqrt{h})$ and the exponential decay of the heat kernel, we have that there exists a bounded function $\omega: \mathbf{R}^k \to \mathbf{R}$ such that
\begin{equation}\label{eq:graph}
\frac{1}{2} + \omega((0, z(x)))\sqrt{h} = \int_{\mathbf{R}^{k-1}} \int_{-\infty}^{\zeta(y)+z(x)} \frac{e^{-\frac{|y|^2+|s|^2}{4h}}}{(4\pi h)^{k/2}}ds dy.
\end{equation}
Recalling that the Gaussian integrates to $1/2$ over half-spaces, we get that \eqref{eq:graph} reads
\begin{equation}\label{eq:to_evaluate}
\omega((0,z(x)))\sqrt{h} = \int_{\mathbf{R}^{k-1}} \int_{0}^{\zeta(y)+z(x)} \frac{e^{-\frac{|y|^2+|s|^2}{4h}}}{(4\pi h)^{k/2}}ds dy.
\end{equation}
Since $\zeta(0) = 0$ and $D\zeta(0) = y$, there exists a bounded function $\zeta_1$ such that $\zeta(v) = \zeta_1(v)|v|^2$. We also observe that
\begin{equation}
e^{-t} \ge 1 - t,\ t \ge 0.
\end{equation}
In particular
\begin{align*}
\omega((0,z(x)))\sqrt{h} &\ge \frac{1}{(4\pi h)^{k/2}}\int_{\mathbf{R}^{k-1}}e^{-\frac{|y|^2}{4h}} \int_{0}^{\zeta(y)+z(x)} \left(1 -\frac{s^2}{4h}\right)ds dy
\\ &=\frac{1}{(4\pi h)^{k/2}}\int_{\mathbf{R}^{k-1}}e^{-\frac{|y|^2}{4h}} \left( \zeta_1(y)|y|^2 + z(x) - \frac{1}{12h}\left( \zeta_1(y)|y|^2 + z(x)\right)^3 \right) dy.
\end{align*}
By using the change of variable $y \to \sqrt{h}y$ we obtain
\begin{equation}
\omega((0,z(x))\sqrt{h} \ge \frac{1}{h^{1/2}}\left( z(x) + \frac{q_1}{h} z(x)^3 + q_2h + q_3h^2+q_4z(x)^2\right),
\end{equation}
where $q_1, q_2, q_3, q_4$ are coefficients depending on the first six moments of the function $y \mapsto e^{-|y|^2}$. By multiplying both sides by $\sqrt{h}$ we get
\begin{equation}\label{eq:final_est}
\omega((0, z(x)) h - \left(q_2h + q_3 h^2 + q_4z(x)^2\right) \ge z(x) + \frac{q_1}{h}z(x)^3.
\end{equation}
By applying \cite[Lemma 6.1]{Fuchs2022} (which holds true even if we additionally consider a bounded drift $\omega$), we have that $z(x) = O(h^{3/2})$. In particular, for $h$ small enough
\begin{equation}
\frac{1}{2} < 1-\frac{q_1}{h}z(x)^2,
\end{equation}
in other words
\begin{equation}
2\omega((0, z(x)) h - 2\left(q_2h + q_3 h^2 + q_4z(x)^2\right) \ge z(x) \ge 0,
\end{equation}
from which we conclude that $z(x) = O(h)$.
\end{proof}

\begin{proof}[Proof of Corollary \ref{corollary:thresholding_ball}]
Denote by $\tilde{C}_{r, x_0}$ the constant obtained by applying Theorem~\ref{thm:thresholding_one_step} to $\Omega_0 = B_{r}(x_0)$. Since $\tilde{C}_{r, x_0}$ depends on $\Omega_0$ only through the $C^0$ norm of the second fundamental form $S_{r,x_0}$ of $\partial B_{r}(x_0)$, it is sufficient to show that this can be bounded independently of $\frac{R}{2} \le r \le R$ and $x_0 \in M$. We clearly have that
\begin{equation}
(0, \operatorname{diam}(M)) \times M \ni (r, x_0) \mapsto \Vert S_{r, x_0} \Vert_{C^0}
\end{equation}
is a continuous function. It is thus bounded on the compact set
\begin{equation}
W := \left\{ (r, x) \in (0,+\infty) \times M: \frac{R}{2} \le r \le R, x \in M \right\}.\qedhere
\end{equation}
\end{proof}

\begin{proof}[Proof of Theorem~\ref{thm:consist}]
For ease of notation, let us assume that $\kappa = 1$. We start by observing that
\begin{equation}
\int_{M \setminus B_{h_n^{\frac{1}{4}}}(z_{h_n})} H(h_n, z_{h_n}, y) \xi(y) d\volm(y) = o\left(\sqrt{h_n}\right).
\end{equation}
This is proved by using the Gaussian bounds from Theorem~\ref{thm:gaussian_bounds}, as we did in \cite[Theorem 3, Step 2]{LauxLelmiI}. In particular, both in \eqref{item:nondeg_grad} and in \eqref{item:deg_grad} of Theorem~\ref{thm:consist} we can replace the domain of integration with
\begin{equation}
\left\{\psi_{h_n}(s_{h_n} - h_n, \cdot) \ge 0\right\} \cap B_{h_n^{\frac{1}{4}}}(z_{h_n}).
\end{equation}
In this way, the sequence of integrals can be computed in normal coordinates around $z_{h_n}$, i.e.,
\begin{align*}
&\int_{\left\{\psi_{h_n}(s_{h_n} - h_n, \cdot) \ge 0\right\} \cap B_{h_n^{\frac{1}{4}}}(z_{h_n})} H(h_n, z_{h_n}, y) \xi(y) d\volm(y)
\\ &= \int_{\left\{\tilde{\psi}_{h_n}(s_{h_n} - h_n, \cdot) \ge 0\right\} \cap B_{h_n^{\frac{1}{4}}}(0)} H(h_n, z_{h_n}, \operatorname{exp}_{z_{n_n}}(y)) \xi(\operatorname{exp}_{z_{n_n}}(y)) \sqrt{\operatorname{det}(g)} dy,
\end{align*}
where we set
\begin{equation}
\tilde{\psi}_{h_n}(t, y) := \psi_{h_n}(t, \exp_{z_{h_n}}(y)),\ y \in B_{\frac{\operatorname{inj}(M)}{2}}(0).
\end{equation}
Using the asymptotic expansion for the heat kernel in Theorem~\ref{thm:asymptotic_exp}, it is easy to see that
\begin{align*}
&\begin{aligned}\int_{\left\{\tilde{\psi}_{h_n}(s_{h_n} - h_n, \cdot) \ge 0\right\} \cap B_{h_n^{\frac{1}{4}}}(0)} H(h_n, z_{h_n}, \operatorname{exp}_{z_{n_n}}(y)) \xi(\operatorname{exp}_{z_{n_n}}(y)) \sqrt{\operatorname{det}(g)} dy
\end{aligned}
\\ &\begin{aligned}
= &\int_{\left\{\tilde{\psi}_{h_n}(s_{h_n} - h_n, \cdot) \ge 0\right\} \cap B_{h_n^{\frac{1}{4}}}(0)} \frac{e^{-\frac{|y|^2}{4h_n}}}{(4\pi h_n)^{k/2}}v_0(z_{h_n}, \operatorname{exp}_{z_{h_n}}(y)) \xi(\operatorname{exp}_{z_{n_n}}(y)) \sqrt{\operatorname{det}(g)} dy
\\ &+ o(\sqrt{h}_n).
\end{aligned}
\end{align*}
In particular, in both \eqref{item:nondeg_grad} and \eqref{item:deg_grad} in Theorem~\ref{thm:consist} the integrals may be substituted with 
\begin{align*}
\int_{\left\{\tilde{\psi}_{h_n}(s_{h_n} - h_n, \cdot) \ge 0\right\} \cap B_{h_n^{\frac{1}{4}}}(0)} \frac{e^{-\frac{|y|^2}{4h_n}}}{(4\pi h_n)^{k/2}}v_0(z_{h_n}, \operatorname{exp}_{z_{h_n}}(y)) \xi(\operatorname{exp}_{z_{n_n}}(y)) \sqrt{\operatorname{det}(g)} dy.
\end{align*}
These integrals may be furthermore decomposed into the sums $\mathbb{I}_n + \mathbb{II}_n$,
\begin{align}
&\begin{aligned}
\mathbb{I}_n := \int_{\left\{\tilde{\psi}_{h_n}(s_{h_n} - h_n, \cdot) \ge 0\right\} \cap B_{h_n^{\frac{1}{4}}}(0)} \frac{e^{-\frac{|y|^2}{4h_n}}}{(4\pi h_n)^{k/2}}dy,
\end{aligned}
\\ &\begin{aligned}
\mathbb{II}_n :=\ &\int_{\left\{\tilde{\psi}_{h_n}(s_{h_n} - h_n, \cdot) \ge 0\right\} \cap B_{h_n^{\frac{1}{4}}}(0)} \frac{e^{-\frac{|y|^2}{4h_n}}}{(4\pi h_n)^{k/2}}(w_n(y)-1)dy,
\end{aligned}
\end{align}
where we define
\begin{equation}
w_n(y) := v_0(z_{h_n}, \operatorname{exp}_{z_{h_n}}(y)) \xi(\operatorname{exp}_{z_{n_n}}(y)) \sqrt{\operatorname{det}(g)}.
\end{equation}
We claim that
\begin{equation}\label{eq:claim_second_int}
\lim_{n \to +\infty} \mathbb{II}_n = \begin{cases}
0\ &\text{if}\ \nabla \psi(s, z) = 0,
\\ \frac{1}{2\sqrt{\pi}|\nabla \psi(s, z)|} \langle \frac{\nabla \xi}{\xi}(z), \nabla \psi(s, z) \rangle\ &\text{otherwise}.
\end{cases}
\end{equation}
Using \eqref{eq:first_coefficient} we see that
\begin{equation}
w_n(y) = \sqrt{\frac{\xi(\operatorname{exp}_{z_{n_n}}(y))\operatorname{det}(g)}{\xi(z_{h_n}) \operatorname{det}(d_{\operatorname{exp}_{z_{h_n}}^{-1}(y)}(\operatorname{exp}_{z_{h_n}}))}}.
\end{equation}
In particular, denoting $\tilde{\xi}_n = \xi \circ \operatorname{exp}_{z_{h_n}}$ and $D_n := \operatorname{det}(d_{\operatorname{exp}_{z_{h_n}}^{-1}(y)}(\operatorname{exp}_{z_{h_n}}))$ we get
\begin{equation}
Dw_n = \frac{1}{2w_n(y)}\frac{\left( (D_y\tilde{\xi}_n)\operatorname{det}(g) + \tilde{\xi}_n D_y\operatorname{det}(g))\tilde{\xi}_n(0) D_n - \tilde{\xi}_n \operatorname{det}(g) \tilde{\xi}_n(0)D_yD_n \right)}{\tilde{\xi}_n(0)^2 D_n^2}.
\end{equation}
We now recall that, in normal coordinates $g(z_{h_n}) = Id$, $Dg(z_{h_n}) = 0$, in particular
\begin{equation}
Dw_n(z_{h_n}) = \frac{1}{2}\frac{D\tilde{\xi}_n}{\tilde{\xi}_n}(0),
\end{equation}
and by a Taylor expansion
\begin{equation}
Dw_n(y) = 1 + \frac{1}{2}\frac{D\tilde{\xi}_n}{\tilde{\xi}_n}(0) \cdot y + O(|y|^2);
\end{equation}
in particular, we infer that
\begin{equation}
\mathbb{II}_n = \frac{1}{2}\frac{D\tilde{\xi}_n}{\tilde{\xi}_n}(0) \cdot \int_{\left\{\tilde{\psi}_{h_n}(s_{h_n} - h_n, \cdot) \ge 0\right\} \cap B_{h_n^{\frac{1}{4}}}(0)} \frac{e^{-\frac{|y|^2}{4h_n}}}{(4\pi h_n)^{k/2}}y dy + O(h_n).
\end{equation}
Now we claim that
\begin{align}\label{eq:reducted_claim}
&\lim_{n \to +\infty} \frac{1}{2\sqrt{h_n}}\frac{D\tilde{\xi}_n}{\tilde{\xi}_n}(0) \cdot \int_{\left\{\tilde{\psi}_{h_n}(s_{h_n} - h_n, \cdot) \ge 0\right\} \cap B_{h_n^{\frac{1}{4}}}(0)} \frac{e^{-\frac{|y|^2}{4h_n}}}{(4\pi h_n)^{k/2}}y dy
\\ &= \frac{1}{2\sqrt{\pi}|\nabla \psi(s, z)|} \frac{D\tilde{\xi}}{\tilde{\xi}}(0) \cdot D\tilde{\psi}(s,0),\nonumber
\end{align}
where $\tilde{\xi} = \xi \circ \operatorname{exp}_z$. Of course \eqref{eq:reducted_claim} gives \eqref{eq:claim_second_int}. 

To see that \eqref{eq:reducted_claim} holds, we start by changing variable in the integral by setting $y = \frac{y}{\sqrt{h_n}}$, which gives that the argument in the limit equals
\begin{align*}
\frac{1}{2}\frac{D\tilde{\xi}_n}{\tilde{\xi}_n}(0) \cdot \int_{\left\{y|\ \tilde{\psi}_{h_n}(s_{h_n} - h_n, \sqrt{h_n}y) \ge 0\right\} \cap B_{h_n^{-\frac{1}{4}}}(0)} \frac{e^{-\frac{|y|^2}{4}}}{(4\pi)^{k/2}}y dy.
\end{align*}
We now let $R_n$ be a sequence of orthogonal matrices such that $R_n^T e_1 = \frac{D\tilde{\xi}_n(0)}{|D\tilde{\xi}_n(0)|}$ and without loss of generality we assume that the sequence converges to an orthogonal matrix $R$. We change variable by setting $y = R_n^Ty$ and we get that the argument of the limit becomes
\begin{align*}
\frac{|D\tilde{\xi}_n(0)|}{2}\int_{\mathcal{C}_n \cap B_{h_n^{-\frac{1}{4}}}(0)} \frac{e^{-\frac{|y|^2}{4}}}{(4\pi)^{k/2}}y_1 dy,
\end{align*}
where we define
\begin{align*}
\mathcal{C}_n:= \left\{y \in \mathbf{R}^k |\ \tilde{\psi}_{h_n}(s_{h_n} - h_n, R_n\sqrt{h_n}y) \ge 0\right\}.
\end{align*}
We now observe that, by Taylor expanding $\tilde{\psi}_{h_n}(t_{h_n} - \cdot, \cdot)$ around $(0, 0)$
\begin{align*}
\tilde{\psi}_{h_n}(s_{h_n} - h_{n}, R_n\sqrt{h_n} y) = &\delta_{h_n} + \sqrt{h_n}R_n^TD\tilde{\psi}_{h_n}(s_{h_n}, 0)  \cdot y 
\\ &- h_n \partial_s\tilde{\psi}_{h_n}(s_{h_n}, 0) + o(|y|^2 + h_n^2),
\end{align*}
thus
\begin{align*}
\mathcal{C}_n = \bigg\{ y \in \mathbf{R}^k|\ &\frac{\delta_{h_n}}{\sqrt{h_n}} + R_n^TD\tilde{\psi}_{h_n}(s_{h_n}, 0)  \cdot y 
\\ &- \sqrt{h_n} \partial_s\tilde{\psi}_{h_n}(s_{h_n}, 0) + o(\sqrt{h_n}|y|^2 + h_n^{\frac{3}{2}}) \ge 0 \bigg\}.
\end{align*}
Recalling assumption \eqref{eq:delta_conv_rate} this re-reads
\begin{align*}
\mathcal{C}_n = \bigg\{ y \in \mathbf{R}^k|\ & R_n^TD\tilde{\psi}_{h_n}(s_{h_n}, 0)  \cdot y + o(1) \ge 0 \bigg\}.
\end{align*}
Observe also that
\begin{align*}
R_nD\tilde{\xi}_n(0) &= |D\tilde{\xi}_n(0)|e_1 
\\ &= \sqrt{\langle \nabla \xi (z_{h_n}) , \nabla \xi (z_{h_n})\rangle}e_1 \underset{n \to +\infty}{\to} \sqrt{\langle \nabla \xi (z) , \nabla \xi (z)\rangle}e_1,
\end{align*}
but also
\begin{align*}
R_nD\tilde{\xi}_n(0) = D\tilde{\xi} \circ R_n^T(0) = D\xi \circ \operatorname{exp}_{z_{h_n}} \circ R_n^T(0)) \underset{n \to +\infty}{\to} RD(\xi \circ \exp_z)(0).
\end{align*}
In other words we must have $D\tilde{\xi}(0) = |D\tilde{\xi}(0)|R^Te_1$. In particular

\begin{align*}
\lim_{n \to +\infty} \frac{|D\tilde{\xi}_n(0)|}{2}\int_{\mathcal{C}_n \cap B_{h_n^{-\frac{1}{4}}}(0)} \frac{e^{-\frac{|y|^2}{4}}}{(4\pi)^{k/2}}y_1 dy & = \frac{|D\tilde{\xi}(0)|}{2}\int_{\left\{y|\ R^TD\tilde{\psi}(s, 0)\cdot y\ge 0 \right\}} \frac{e^{-\frac{|y|^2}{4}}}{(4\pi)^{k/2}}y_1 dy
\\ & = \frac{|D\tilde{\xi}(0)|}{2}\int_{\left\{y|\ D\tilde{\psi}(s, 0)\cdot y\ge 0 \right\}} \frac{e^{-\frac{|y|^2}{4}}}{(4\pi)^{k/2}}Ry \cdot e_1 dy
\\ & =  \frac{1}{2}\frac{D\tilde{\xi}}{\tilde{\xi}}(0) \cdot \int_{\left\{y|\ D\tilde{\psi}(s, 0)\cdot y\ge 0 \right\}} \frac{e^{-\frac{|y|^2}{4}}}{(4\pi)^{k/2}} y dy.
\end{align*}
If $\nabla \psi(t, z) = 0$, then the last integral is zero, being component-wise the integral over the whole space of on odd-function. Otherwise we change variable according to $y = O^Ty$, where $O$ is an orthogonal matrix such that $OD\tilde{\psi}(s,0) = |D\tilde{\psi}(s,0)|e_1$, which gives that the last integral equals
\begin{align*}
\frac{1}{2}\frac{OD\tilde{\xi}}{\tilde{\xi}}(0) \cdot \int_{\left\{y|\ y_1\ge 0 \right\}} \frac{e^{-\frac{|y|^2}{4}}}{(4\pi)^{k/2}} y dy & = \frac{1}{2}\frac{OD\tilde{\xi}}{\tilde{\xi}}(0) \cdot e_1\frac{1}{\sqrt{\pi}} 
\\ &=  \frac{1}{2\sqrt{\pi}|D\tilde{\psi}(s, 0)|}\frac{D\tilde{\xi}}{\tilde{\xi}}(0) \cdot D\tilde{\psi}(s, 0).
\end{align*}
We are now in a position to prove \eqref{item:nondeg_grad} and \eqref{item:deg_grad} in Theorem~\ref{thm:consist}.

Item \eqref{item:nondeg_grad}. By the discussion above, the left hand side of \eqref{eq:claim_non_deg} may be substituted with

\begin{align*}
&\liminf_{n \to +\infty} \frac{1}{\sqrt{h_n}}\left(\frac{1}{2} - \mathbb{I}_n - \mathbb{II}_n\right) 
\\ &\ge \liminf_{n\to +\infty} \frac{1}{\sqrt{h_n}}\left( \frac{1}{2} - \mathbb{I}_n\right) - \frac{1}{2\sqrt{\pi}|\nabla \psi(s, z)|} \langle \frac{\nabla \xi}{\xi}(z), \nabla \psi(s, z) \rangle,
\end{align*}
where we used \eqref{eq:claim_second_int} in the second line. To estimate 
\begin{align*}
 \liminf_{n\to +\infty} \frac{1}{\sqrt{h_n}}\left( \frac{1}{2} - \mathbb{I}_n\right)
\end{align*}
we can use \cite[Proposition 4.1]{Barles1995} applied with
\begin{align*}
&(t_h, x_h) = (s_h, 0),
\\ &(t, x) = (s, 0),
\\ & \phi_h(t, \cdot) = \tilde{\psi}_{h}(t, \cdot).
\end{align*}
The only difference is that here we do not assume that $\phi(t_h, x_h) = 0$, but $\phi(t_h, x_h) = o(\sqrt{h})$ - one can check that the result holds true also with this modification by the same proof of  \cite[Proposition 4.1]{Barles1995}. In particular, we get

\begin{align*}
&\liminf_{n \to +\infty} \frac{1}{\sqrt{h_n}}\bigg(\frac{1}{2} - \int_{\left\{\psi_{h_n}(t_{h_n} - h_n, \cdot) \ge 0\right\}} H(h_n, z_{h_n}, y) \xi(y)d\volm\bigg)
\\ &\ge \frac{1}{2\sqrt{\pi}|D\tilde{\psi}(s, 0)|}\left(\partial_t \tilde{\psi} + \Delta \tilde{\psi} - \frac{D\tilde{\psi} \cdot D^2\tilde{\psi}D\tilde{\psi}}{|D\tilde{\psi}|^2} - \frac{D\tilde{\xi}}{\tilde{\xi}} \cdot D\tilde{\psi} \right),
\end{align*}
which is equal to the right hand side of \eqref{eq:claim_non_deg} because we are using exponential coordinates around $z$ (recall our convention $\Delta = - \sum_{i=1}^k \partial^2_{ii}$).

Item \eqref{item:deg_grad}. Once again, by the above discussion, we can assume that

\begin{equation}
\frac{1}{2} - \mathbb{I}_n \le o(\sqrt{h_n}),
\end{equation}
and the result follows by applying \cite[Proposition 4.1]{Barles1995} with
\begin{align*}
&(t_h, x_h) = (s_h, 0),
\\ &(t, x) = (s, 0),
\\ & \phi_h(t, \cdot) = \tilde{\psi}_{h}(t, \cdot).
\end{align*}
In this case, there are two differences from the original version \cite[Proposition 4.1]{Barles1995}. First of all, we again do not assume that $\phi_h(t_h, x_h) = 0$, but we assume $\phi_h(t_h, x_h) = o(\sqrt{h})$. Then, we assume that  $\frac{1}{2} - \mathbb{I}_n \le o(\sqrt{h_n})$ and not the stronger $\frac{1}{2} - \mathbb{I}_n \le 0$. But a quick inspection of the proof of \cite[Proposition 4.1]{Barles1995} reveals that these changes are irrelevant for the argument to work.
\end{proof}

\section{Appendix}\hypertarget{sec:appendix}{}

\subsection{Results on weighted manifolds}

Hereafter we collect some results about weighted Laplacians and heat kernels on closed manifolds. Let $(M, g)$ be a $k$-dimensional, compact Riemannian manifold endowed with a measure $\mu := \xi \volm$, with $\xi \in C^{\infty}(M)$, $\xi > 0$. We denote by $\Delta_{\xi}$ the associated Laplacian, which is defined on $f \in C^{\infty}(M)$ as

\begin{align*}
\Delta_{\xi}f := -\frac{1}{\xi}\operatorname{div}\left(\xi\nabla f\right).
\end{align*}
We denote by $H$ the corresponding heat kernel, i.e., $H$ is a real valued function defined on $(0, +\infty) \times M \times M$ such that for any $u \in L^2(M)$ the function

\begin{equation}
e^{-t\Delta_{\xi}}u(x) := T(t)u(x) = \int_M H(t, x, y)u(y)d\mu(y),\nonumber
\end{equation}
defined for $(t, x) \in (0, +\infty) \times M$, is the unique solution to the Cauchy problem

\begin{equation}
\begin{cases}
\partial_t v = -\Delta_{\xi}v\ &\text{in}\ (0, +\infty) \times M,
\\
v(0, x) = u(x) &\text{on}\ M,
\end{cases}\nonumber
\end{equation}
where the initial value at $t=0$ is attained in the sense that

\begin{equation}
\lim_{t\downarrow 0} e^{-t\Delta_{\xi}}u = u\ \text{in}\ L^2(M).
\end{equation}
We will use the following results.

\begin{theorem}\label{thm:controlling_norm_ef}
Let $M$, $\xi$ be as above. Let $f$ be an $L^2(\xi)$-normalized eigenfunction of $\Delta_{\xi}$ corresponding to the eigenvalue $\lambda$, then for each integer $m \ge 0$
\begin{equation}
\Vert f \Vert_{C^m(M)} \le C_{M, m} \left( \lambda^{m+1+\frac{k}{2}} + 1\right).
\end{equation}
\end{theorem}

\begin{theorem}\label{thm:gaussian_bounds}
Let $M$, $\xi$ be as above. There exists constants $Q_1, Q_2, Q_3, Q_4, \hat{Q}_1, \hat{Q}_2 > 0$ such that for every $t > 0$ and all $x, y \in M$,
\begin{align}\label{eq:gauss_up_bds_I}
\frac{Q_1}{\mu(B_{\sqrt{t}}(x))} e^{-\frac{d_M^2(x,y)}{Q_2t}} \le H(t,x,y) \le \frac{Q_3}{\mu(B_{\sqrt{t}}(x))} e^{-\frac{d_M^2(x,y)}{Q_4t}}. 
\end{align}
\begin{align}\label{eq:gauss_up_bds_II}
|\nabla_x H(t,x,y)| \le \frac{\hat{Q}_1}{\sqrt{t} \mu(B_{\sqrt{t}}(x))} \operatorname{exp}\left( -\frac{d_M^2(x,y)}{\hat{Q}_2 t}\right).
\end{align}
\end{theorem}

\begin{theorem}\label{thm:asymptotic_exp}
Let $M$, $\xi$ be as above. There exist functions $v_j \in C^{\infty}(M \times M), j \in \mathbf{N}$, such that for every $N> l + \frac{k}{2}$ there exists a constant $\tilde{C}_N < \infty$ such that
\begin{equation}\label{eq:expansion}
\left| \nabla^l \left( H(t, x, y) - \frac{e^{-\frac{d_M^2(x,y)}{4t}}}{(4\pi t)^{k/2}}\sum_{j=0}^{N}v_j(x,y)t^j\right) \right| \le \tilde{C}_N t^{N+1-\frac{k}{2}},
\end{equation} 
provided $d(x,y) \le \frac{\operatorname{inj}(M)}{2}$. Moreover we have
\begin{equation}\label{eq:first_coefficient}
v_0(x,y) = \frac{1}{\sqrt{\xi(x) \xi(y) \operatorname{det}(d_{\operatorname{exp}^{-1}_{x}(y)}\operatorname{exp}_{x})}} .
\end{equation}
\end{theorem}

Theorem~\ref{thm:controlling_norm_ef} follows by the Sobolev embedding theorem and the $L^2$-regularity theory for elliptic equations on manifolds. Theorem~\ref{thm:gaussian_bounds} follows from the Li--Yau inequality for weighted manifolds \cite{Setti1992}. The asymptotic expansion in Theorem~\ref{thm:asymptotic_exp} follows by constructing the heat kernel by means of the \emph{parametrix} method: this construction is technical and we refer to \cite{Rosenberg1997}, where this is carried out for the case $\xi = 1$. Here we just sketch the first part of the construction for a general density $\xi$, which gives \eqref{eq:first_coefficient}. The idea is that when $x, y$ are close enough, say $d(x, y) < \frac{\operatorname{inj}(M)}{2}$, a good approximation for the heat kernel should be given by

\begin{equation}\label{eq:ansatz}
H_N(t, x, y) := G_t(x, y)\left( v_0(x,y) + . . . + t^Nv_N(x, y)\right),
\end{equation}
for smooth functions $v_j$ and $t > 0$. Here
\begin{equation}
G_t(x, y) := \frac{e^{-\frac{d_M^2(x,y)}{4t}}}{(4\pi t)^{k/2}}.
\end{equation}
Since the Ansatz \eqref{eq:ansatz} should be an approximation of the heat kernel, we would like to have
\begin{align}\label{eq:param_goal}
0 = \partial_t H_N + \Delta_{\xi}H_N,
\end{align}
where $\Delta_\xi$ denotes the weighted Laplacian with respect to the $y$-variable. 
We now compute the right hand side of the above equation by using exponential coordinates around $x$: we denote them by $(r, \theta) \in [0,R) \times \mathbb{S}^{k-1}$. Observe that

\begin{align*}
\partial_t H_N &= \partial_t G_t (v_0 + . ..  + t^N v_N) + G_t (v_1 + . . .  + Nt^{N-1}v_N)
\\ & = \left( \frac{r^2}{4t^2} - \frac{k}{2t}\right) G_t(v_0 + . ..  + t^N v_N) + G_t (v_1 + . . .  + Nt^{N-1}v_N).
\end{align*}
Furthermore
\begin{align*}
\Delta_{\xi}H_N =\ &G_t\left( \Delta_{\xi}v_0 + . . . + t^N\Delta_{\xi}v_N\right) 
\\ &+ \Delta_{\xi}G_t (v_0 + . . .  +t^Nv_N) - 2\langle \nabla G_t, \left( \nabla v_0 + . . . + t^N\nabla v_N\right)\rangle.
\end{align*}
Using Gauss' Lemma and the fact that $G_t$ is independent of $\theta$ we get
\begin{align*}
 2\langle \nabla G_t, \left( \nabla v_0 + . . . + t^N\nabla v_N\right)\rangle &= 2 \partial_r G_t (\partial_r v_0 + . . .  +t^N \partial_rv_N)
 \\ &= -\frac{r}{t}G_t (\partial_r v_0 + . . .  +t^N \partial_rv_N).
\end{align*}
We also observe that by definition of $\Delta_{\xi}$ and by using again Gauss' Lemma and the independence of $G_t$ from $\theta$
\begin{align*}
\Delta_{\xi} G_t &= \Delta G_t - \langle \frac{\nabla\xi}{\xi}, \nabla G_t\rangle = \Delta G_t +\frac{r}{2t} \frac{\partial_r\xi}{\xi}G_t.
\end{align*}
We define
\begin{equation}
D(y) :=\operatorname{det}(d_{\operatorname{exp}^{-1}_{x}(y)}\operatorname{exp}_{x}).
\end{equation}
Using the expression of the Laplacian in spherical coordinates and the invariance of $G_t$ with respect to $\theta$ we get

\begin{align*}
\Delta G_t = -\frac{\partial^2G_t}{\partial r^2} - \partial_rG_t\left(\frac{\partial_rD}{D} + \frac{k-1}{r}\right) = -\left( \frac{r^2}{4t^2} - \frac{k}{2t}\right) G_t + \frac{r}{2t}\frac{\partial_r D}{D}G_t.
\end{align*}
Putting things together we have
\begin{align}\label{eq:final_comp_approx}
\partial_t H_N + \Delta_{\xi}H_N = G_t \bigg(& (v_1 + . . . + Nt^{N-1}v_N) - (\Delta_{\xi}v_0 + . . . t^N \Delta_{\xi}v_N)
\\ & + \frac{r}{2t}\partial_r \operatorname{log}(D\xi) (v_0 + . . . + t^Nv_N) + \frac{r}{t} (\partial_r v_0 + . . . + \partial_r v_N))\bigg).\nonumber
\end{align}
Although we cannot get \eqref{eq:param_goal} exactly, we can choose $v_j$ in such a way that 
\begin{equation}
\partial_t H_N + \Delta_{\xi}H_N = G_t t^N \Delta_{\xi} v_N.
\end{equation}
In other words, we choose the coefficients in such a way that

\begin{align}
\frac{r}{2t}\partial_r\operatorname{log}(D\xi) v_0 + \frac{r}{t}\partial_r v_0 &= 0,\label{eq:first}
\\ jt^{j-1}v_j - t^{j-1}\Delta_{\xi}v_{j-1} + t^{j-1}\frac{r}{2}\partial_r\operatorname{log}(D\xi)v_j + rt^{j-1}\partial_rv_j &= 0,\ \text{for}\ 1 \le j \le N.\label{eq:higher}
\end{align}
Once one solves \eqref{eq:first}, one can show inductively that \eqref{eq:higher} admits a smooth solution $v_j$. It is easily seen that \eqref{eq:first} can be solved explicitly to give

\begin{equation}
v_0(x,y) =  \frac{1}{\sqrt{\xi(x) \xi(y) \operatorname{det}(d_{\operatorname{exp}^{-1}_{x}(y)}\operatorname{exp}_{x})}} .
\end{equation}
From here, the construction of the heat kernel and the estimate \eqref{eq:expansion} follow verbatim as in \cite{Rosenberg1997}.

\subsection{Results on random geometric graphs}

In this subsection we use the setting and the notation of Section \ref{sec:mainres}, with the points $\{x_i\}_{i=1}^{+\infty}$ being given by i.i.d.\ random points on $M$, distributed according to a probability distribution $\nu = \rho\volm \in \mathcal{P}(M)$, with $\rho \in C^{\infty}(M)$, $\rho > 0$. The following two results are proved in \cite{Calder2022, Calder2022a} for the unnormalized graph Laplacian, but the proof of the statements extends when we work with the random walk Laplacian. Hereafter, given $l \in \mathbf{N}$, we set
\begin{equation}
    \gamma_{l} := \inf_{j < l, j \in \mathbf{N}} (\lambda_{j+1} - \lambda_j).
\end{equation}

\begin{theorem}\label{thm:eigv_eigf_est}
In the above-mentioned setting, if additionally, the eigenvalues of $\Delta_{\rho^2}$ are simple, then for every $\overline{l} \in \mathbf{N}$ we have that with probability greater than 
\begin{align*}
1 - Q_1 \epsilon_n^{-6k}\operatorname{exp}(-Q_2n\epsilon_n^{k+4}) - Q_3n\operatorname{exp}(-Q_4n\left( \lambda_{\overline{l}} + 1 \right)^{-k})
\end{align*}
we have for every $l \le \overline{l}$
\begin{equation}
|\lambda_n^l - \kappa(\eta)\lambda_l| +  \max_{z \in V_n} \left| v_n^l(z) - \frac{f_l(z)}{C_1^{1/2}}\right| \le Q_5 \frac{\Vert f_l \Vert_{C^3(M)}}{\gamma_{l}}\epsilon_n.
\end{equation}
\end{theorem}
\begin{theorem}\label{thm:density_est}
In the above-mentioned setting, if $n$ is large enough, with probability greater than $1-Q_6\epsilon_n^{-k}\operatorname{exp}(-Q_7n \epsilon_n^{k+2})$,
we have that
\begin{equation}
 \max_{z \in V_n} |d_{n, \epsilon_n}(z) - C_1 \rho(z)| \le Q_8 \epsilon_n.
\end{equation}
\end{theorem}

We also recall the following result, which may be easily derived from \cite[Theorem 2]{GarciaTrillos2020}.

\begin{theorem}\label{thm:transp_plan_existence}
Let $(M, g)$ be a $k$-dimensional closed Riemannian manifold. Let $\rho \in C^{\infty}(M)$, $\rho > 0$ such that $\nu := \rho \volm \in \mathcal{P}(M)$. Let $\{X_i\}_{i \in \mathbf{N}}$ be i.i.d.\ random points in $M$ distributed according to $\nu$ and let $\nu_n := \frac{1}{n}\sum_{i=1}^n \delta_{X_i}$ be the associated empirical measures. Then there is a constant $C > 0$ such that almost surely there exist transport maps $T_n$ such that $(T_n)_{\#} \nu = \nu_n$ and
\begin{align}\label{eq:cond_transp_plans}
\begin{cases}
\limsup_{n \to +\infty} \frac{n^{1/2}\sup_{x \in M} d_M(x, T_n(x))}{log^{3/4}(n)} \le C \quad \text{if}\ k = 2,
\\ \limsup_{n \to +\infty} \frac{n^{1/k}\sup_{x \in M} d_M(x, T_n(x))}{log^{1/k}(n)} \le C \quad \text{if}\ k \ge 3.
\end{cases}
\end{align}
\end{theorem}

\section*{Acknowledgements}

This project has received funding from the Deutsche Forschungsgemeinschaft (DFG, German Research Foundation) under Germany's Excellence Strategy -- EXC-2047/1 -- 390685813.

\bibliography{bib_first_draft}{}
\bibliographystyle{plain}

\end{document}